\documentclass[leqno,11pt]{amsart}
\usepackage[colorlinks,pagebackref,hypertexnames=false]{hyperref}

\usepackage{amsmath,amsthm,amssymb,mathrsfs} 
\usepackage[alphabetic,backrefs]{amsrefs}
\usepackage{ae,aecompl}
\usepackage[margin=1in]{geometry} 
\usepackage[retainorgcmds]{IEEEtrantools}
\usepackage{float}
\usepackage[onehalfspacing]{setspace}
\usepackage{soul}

\usepackage[english]{babel}

\usepackage{bookmark}

\usepackage{amsmath,amsfonts,amssymb,graphics,graphicx,latexsym,amscd,subfigure,hyperref}
\usepackage[usenames]{xcolor}
\usepackage[retainorgcmds]{IEEEtrantools}
\usepackage{graphicx}
\usepackage{amssymb}
\usepackage{xypic}
\usepackage{hyperref}
\usepackage{amsmath,amssymb}
\usepackage{dsfont}

\numberwithin{equation}{section}

\DeclareMathOperator\Hess{\mathrm{Hess}}

\newcommand{\cE}{\mathcal{E}}
\newcommand{\cM}{\mathcal{M}}
\newcommand{\cF}{\mathcal{F}}

\newtheorem{theorem}{Theorem}

\newtheorem{corollary}{Corollary}
\newtheorem{lemma}{Lemma}
\newtheorem{proposition}{Proposition}

\theoremstyle{definition} 
\newtheorem{example}{Example}

\newtheorem{remark}{Remark}

\title[]{Electrostatics and geodesics on $K3$ surfaces} 

\author{Gon\c{c}alo Oliveira}
\address{Departamento de Matem\'atica, Intituto Superior T\'ecnico, Lisboa, Portugal; and IST Austria, Klosterneuburg, Austria}
\urladdr{\href{https://sites.google.com/view/goncalo-oliveira-math-webpage/home}{sites.google.com/view/goncalo-oliveira-math-webpage/home}}
\email{\href{mailto:galato97@gmail.com}{galato97@gmail.com}}
\date{\today}
\keywords{Closed geodesics, K3 surfaces, Ricci-flat K\"ahler metrics}
\subjclass[2010]{53C22,58E10,53C25,53C26}

\begin{document}

	\begin{abstract}
		Motivated by some conjectures originating in the Physics literature, we use Foscolo's construction of Ricci-flat K\"ahler metrics on K3 surfaces to locate, with high precision, several closed geodesics and compute their index (their length is also approximately known).
		
		Interestingly, the construction of these geodesics is related to an open problem in electrostatics posed by Maxwell in 1873 \cite{Maxwell}. Our construction is also of interest to modern Physicists working on (supersymmetric) non-linear sigma models with target space such a K3 surface \cite{Douglas}.
	\end{abstract}
	
	\maketitle
	
	\tableofcontents

	\section{Introduction}
	
	\subsection{Context}
	
	The problem of constructing closed geodesics on Riemannian manifolds has a long history which dates back at least to Jacobi's study of geodesics on ellipsoids \cite{Jacobi} almost two centuries ago. In the twentieth century, the attempts at constructing and understanding closed geodesics on more general Riemannian manifolds led to the development of several dynamical and analytic techniques respectively pioneered by Poincar\'e, Birkhoff \cite{Poincare,Birkhoff}, and Morse, Lusternik-Schnirelmann \cite{Morse0,LuS}. The later of these techniques consisted in constructing a Morse theoretic framework for the energy functional in the free loop space. Together with new developments in topology \cite{VS,Gromoll}, this eventually led to the proof that any closed Riemannian manifold whose homology group is not a truncated polynomial, admits infinitely many closed geodesics. The proof of this result relies in showing that under such circumstances the homology of the free loop space is infinitely generated (with generators in infinitely large degree). Then, by using a suitable version of Morse theory one concludes that there are infinitely many closed geodesics and that their index grows with the degree of the respective homology class. For a proof of this and other results, the author recommends the book \cite{Klingenberg} and references therein.
	
	Despite such tremendous progress, questions such as explicitly constructing (or locating) geodesics on a specific Riemannian manifold, and determining their length and index, are extremely hard and require very good knowledge of the metric in order to be tackled. Indeed, finding closed geodesics involves solving a system of second order ordinary differential equations and establish the existence of periodic solutions, a problem which is extremely hard, especially if the metric is not known in explicit form. Furthermore, in a way, the study of geodesics on a Riemannian manifold can also be considered an indirect way of studying the underlying geometry.
	
	Specific examples of Riemannian manifolds for which it would be interesting to know more about their closed geodesics are for instance Ricci-flat K\"ahler metrics on K3 surfaces or more general Calabi--Yau manifolds. Much of this interest comes from the need to better understand the geometry of such metrics and from their relevance in Physics. Indeed, according to \cite{Douglas}, closed geodesics should play a particularly important role in the (supersymmetric) non-linear sigma model with target space such a Calabi--Yau manifold. However, tackling such questions seems extremely difficult and out of reach as no Ricci-flat K\"ahler metrics on compact Calabi--Yau manifolds is known in closed form.\footnote{See however \cite{TZ0,TZ1} for some very promising and interesting development in this direction. To our best understanding, the authors of these references find a formal expansion which converges to Ricci-flat K\"ahler metrics on K3 surfaces.} Nevertheless, in certain limiting regimes where the metric either degenerates to an orbifold metric or collapses, it is possible to obtain good approximations to such metrics \cite{LS,Donaldson,GW,Foscolo1,Hein,SZ}. 
	This gives a glimpse of hope in order to search for closed geodesics on such degenerating K3 surfaces. A hope which we explore in this article.

	\subsection{Summary}
	
	Motivated by the discussion above and the proposals of \cite{Douglas} regarding the importance of closed geodesics on Calabi--Yau manifolds for the supersymmetric sigma-model. In this article we consider Ricci-flat K\"ahler metrics nearly collapsing to a quotient of a $3$-torus $\mathbb{T}_\Lambda$ by $\mathbb{Z}_2$. These metrics have been constructed by Lorenzo Foscolo in \cite{Foscolo1} and can be well approximated using the Gibbons--Hawking ansatz. 
	
	Away from the finite set of points at which the curvature concentrates, we use this approximation to investigate which collapsing circles approximate closed geodesics. The location of these circles above the quotient $\mathbb{T}_\Lambda/\mathbb{Z}_2$ is related to a classical problem in electrostatics first posed by Maxwell in his ``Treatise on Electricity and Magnetism'' \cite{Maxwell}. Consider a configuration of point charges symmetric under the $\mathbb{Z}_2$-action and containing its fixed points. If the total charge of this configuration vanishes (a  necessary condition in order to place them in a compact manifold such as $\mathbb{T}_\Lambda$), they generate an electric field $E$ on $\mathbb{T}_\Lambda/\mathbb{Z}_2$, whose electrostatic points, i.e. those points at which $E$ vanishes, correspond to approximate geodesics circles for the collapsing metrics on the K3 surfaces under consideration. 
	
	The other regime we can consider is what happens at the finite set of points at which the curvature concentrates. It turns out that the local models for the (rescaled) metrics at these points are ALF gravitational instantons to which we can, at times, apply a similar relation to electrostatics to locate closed geodesics. 
	
	In this manner, for such collapsing Ricci-flat K\"ahler metrics on K3 surfaces, we are able to locate several closed geodesics and compute their index.

	\subsection{Main result}
	
	
	Before stating any concrete result we must start with a brief description of Foscolo's Ricci-flat K\"ahler metrics on the K3 surface, simply referred to as Foscolo's K3 surfaces in this article. As Riemannian manifolds these form an open set in the moduli space of hyperK\"ahler metrics on the differentiable manifold underlying K3 and so are parameterized by $58$ parameters. One of these, denoted $\epsilon \ll 1$, measures how quickly a circle fibration in a large portion of $X$ is collapsing. Namely, it approximately corresponds to the length of the circle fibers. The resulting Riemannian manifold $X_\epsilon=(X, g_\epsilon^{hk})$ can be decomposed as
	\begin{equation}\label{eq:decomposition of X epsilon INTRO}
		X_\epsilon = K_\epsilon \cup  \bigcup_{j=1}^8 M^j_\epsilon \cup \bigcup_{i=1}^n N^i_\epsilon.
	\end{equation} 
	The \emph{building blocks} in this decomposition consist of the collapsing region $K_\epsilon$, and the curvature concentrated region $\bigcup_{j=1}^8 M^j_\epsilon \cup \bigcup_{i=1}^n N^i_\epsilon$. They have the following list of properties:
		\begin{enumerate}
		\item $K_\epsilon$ is a large open set which is a circle fibration over the quotient of a punctured torus $\mathbb{T}_\Lambda$ by $\mathbb{Z}_2$, with the circle fibers having size approximately $\epsilon$. The punctures in the torus $\mathbb{T}_\Lambda = \mathbb{R}^3/ \Lambda$ must split in two sets $\lbrace q_1 , \ldots , q_8\rbrace$ which correspond to the points with $2q_i \in \Lambda$ and $\lbrace p_1 , \ldots , p_n , -p_1, \ldots , -p_n \rbrace$ the first $n$ of which can be chosen freely and are parameters of the construction. Furthermore, we must choose some extra numbers $m_1, \ldots , m_8 \in \mathbb{N}_0$ and $k_1, \ldots , k_n \in \mathbb{N}$ which together with $n$ must satisfy the balancing condition
		\begin{equation}\label{eq:necessary condition for Laplace equation INTRO}
			\sum_{i=1}^n k_i + \sum_{j=1}^8 m_j =16 .
		\end{equation}
		This is a necessary (and sufficient) condition in order to construct a harmonic function
		$$h:\mathbb{T}_\Lambda \to \mathbb{R}\cup \lbrace - \infty \rbrace \cup \lbrace + \infty \rbrace,$$ 
		which corresponds to the electric potential generated by charges of magnitude $k_i$ placed at the points $p_i$ and $-p_i$ for $i=1, \ldots , n$, and charges of magnitude $2m_j-4$ at the points $q_j$ for $j = 1 , \ldots , 8$. Using this function, the Gibbons--Hawking anstaz determines a metric $g_\epsilon^{gh}$ which we can use to approximate the hyperK\"ahler metric $g_\epsilon^{hk}$ in the K3 surface restricted to this region. We shall denote the open Riemannian manifold determined by the same smooth manifold as $K_\epsilon$, but with the metric $g_\epsilon^{gh}$, by $K_\epsilon^{gh}$.
		
		\item For each $j=1, \ldots , 8$, the rescaled metrics $\epsilon^{-2}g_\epsilon^{hk}$ restricted to $M_\epsilon^j$ smoothly converge to an initially fixed ALF gravitational instanton $M^j$ of type $D_{m_j}$.
		
		\item For each $i=1, \ldots , n$, the rescaled metrics $\epsilon^{-2}g_\epsilon^{hk}$ restricted to $N_\epsilon^i$ smoothly converge to an initially fixed ALF gravitational instanton $N^i$ of type $A_{k_i-1}$. The metric on $N^i$ can be described explicitly using the Gibbons--Hawking anstaz and is determined by a function 
		$$\phi_i: \mathbb{R}^3 \to \mathbb{R}^+ \cup \lbrace \infty \rbrace$$ 
		corresponding to the electric potential generated by $k_i$ unit charges in $\mathbb{R}^3$ whose location $ x^{i}_1 , \ldots , x^{i}_{k_i} $ determines the metric on $N^i$. Away from a finite set of points, $N_i$ consists of a circle fibration over $ \mathbb{R}^3 \backslash \lbrace x^{i}_1 , \ldots , x^{i}_{k_i} \rbrace$ with the circle fibers having length $2 \pi \phi_i^{-1/2}$.
	\end{enumerate}

	Theorem \ref{thm:Main} below, is a more informal version of \ref{thm:Main extended} later in text and serves as a sample of the results we are able to obtain. We state it here in full generality, which weakens some of the conclusions we are able to make when dealing with more specific configurations of points. In order to illustrate those stronger results, later in this introduction we shall examine two concrete examples on which we can do better than the general result.
	
	\begin{theorem}\label{thm:Main}
%
		Let $n \geq 1$, $k_1, \ldots , k_n, m_1, \ldots , m_8$ satisfying \ref{eq:necessary condition for Laplace equation INTRO}, a collection of ``bubbles'' $N^1, \ldots , N^n$, $M^1, \ldots , M^8$, and $p_1, \ldots , p_n \in \mathbb{T}_\Lambda$ with $p_1$ generically chosen. Then, there is a positive number $\epsilon_0 \ll 1$, such that for all $\epsilon < \epsilon_0$: 
		\begin{itemize}
			\item There are at least $n+1$ closed geodesics of index one and $5$ of index two in $K_\epsilon$.
			
			\item For $i=1, \ldots , n$ and if the location of $x^{i}_1$ in the ``bubble'' $N^i$ is generic, there are at least $k_i-1$ closed geodesics on $N_\epsilon^i$. In addition, if $\min_{j_1 \neq j_2}|x^{i}_{j_1} - x^{i}_{j_2}|$ is sufficiently large, these geodesics have index $1$. 
		\end{itemize}
		These geodesics can be located within a precision of $o(\epsilon)$, in the Hausdorff distance, within the circle fibers above the critical points of the corresponding electric potentials.
	\end{theorem}

	Before stating the stronger results we promised, we shall say some words about the proof of this result and how to more precisely locate the geodesics. In order to construct and locate closed geodesics on $X_\epsilon$ we start by locating these for the approximate metrics on the building blocks $K_\epsilon^{gh}$, $M^j$, and $N^i$, which we subsequently deform to closed geodesics on $X_\epsilon$.
	
	On the other hand, to construct closed geodesics for the approximate metrics on the building blocks we investigate which circle fibers are geodesic for the hyperK\"ahler metrics constructed using the Gibbons--Hawking ansatz. It turns out that the circle fibers above critical points of $h$ are closed geodesics on $K_\epsilon^{gh}$ and that the circle fibers above the critical points of each $\phi_i$ are closed geodesics for $N^i$. Given that critical points of these functions correspond to electrostatic points for the resulting electric fields, it is at this point that the relation to electrostatics comes into play. As for the connection with Maxwell's 1873 problem in \cite{Maxwell}, it arises when trying to maximize the number of these geodesics, which corresponds to maximizing the number of electrostatic points. Maxwell's conjecture is that this is at most $(\ell-1)^2$ where $\ell$ is the number of point charges. See \cite{Gabrielov} for the only development towards this conjecture so far.
	
	Now that the geodesics for the approximate metrics have been determined it is time to deform these to the nearby hyperK\"ahler metric. White's deformation theorem in \cite{White} implies that the closed geodesics for $N^i$ can be deformed to closed geodesics for $N^i_\epsilon$ for sufficiently small $\epsilon$. On the other hand, given that the metric in $K_\epsilon$ is collapsing, one cannot use White's deformation theorem to deform the geodesics on $K_\epsilon^{gh}$ to geodesics on $K_\epsilon$ and doing so requires a separate analysis which we perform. This presents an additional difficulty as the Jacobi operator of the approximate geodesics is itself becoming degenerate as $\epsilon \to 0$. Our method to overcome this difficulty results from appropriately splitting the geodesic equation into pieces and performing a careful analysis which shows that each of the remaining terms are themselves converging to zero at a faster rate in $\epsilon$. As for the statement regarding the index, it follows from computing that of the approximate geodesics using the explicit form of the Jacobi operator, which can be written down in terms of the Hessians of $h$ and $\phi_i$ at the relevant points.
	
	Finally, as promised, we give concrete examples, some of which we summarize here. For a more comprehensive treatment, we refer the reader to Section \ref{sec:examples} which contains a more detailed version of these and several other concrete applications of our method. 
	
	\begin{example}
		Set all $m_1 = \ldots = m_8=0$ and consider the partition of $16$ given by $9+3+3+1$, then set $k_1=1$, $k_2=3=k_3$, $k_4=9$ and so $n=4$. According to Theorem \ref{thm:Main}, in this situation, there are at least $n+1=5$ closed geodesics of index one and $5$ of index two in the large open set $K_\epsilon$. These are within a distance of $\epsilon^{13/10}$ from the fibers over the critical points of $h$.
		
		Furthermore, $N^1$ is the Taub-Nut manifold which admits no closed geodesic and so we obtain no such in $N^1_\epsilon$. On the other hand, $N^2=N^3$ are gravitational instantons of type $A_2$ and we select them so that they are obtained from the Gibbons--Hawking anstaz with $\phi_2(x)=\phi_3(x)= 1+ \sum_{l=1}^3 \frac{1}{2|x-x_l|}$ given by the electric potential generated by $3$ electric charges $\lbrace x_1, x_2 ,x_3 \rbrace$ placed in the vertices of an equilateral triangle as in figure \ref{fig:triangle}.

		\begin{figure}[htp]\label{fig:triangle}
			\centering
			\includegraphics{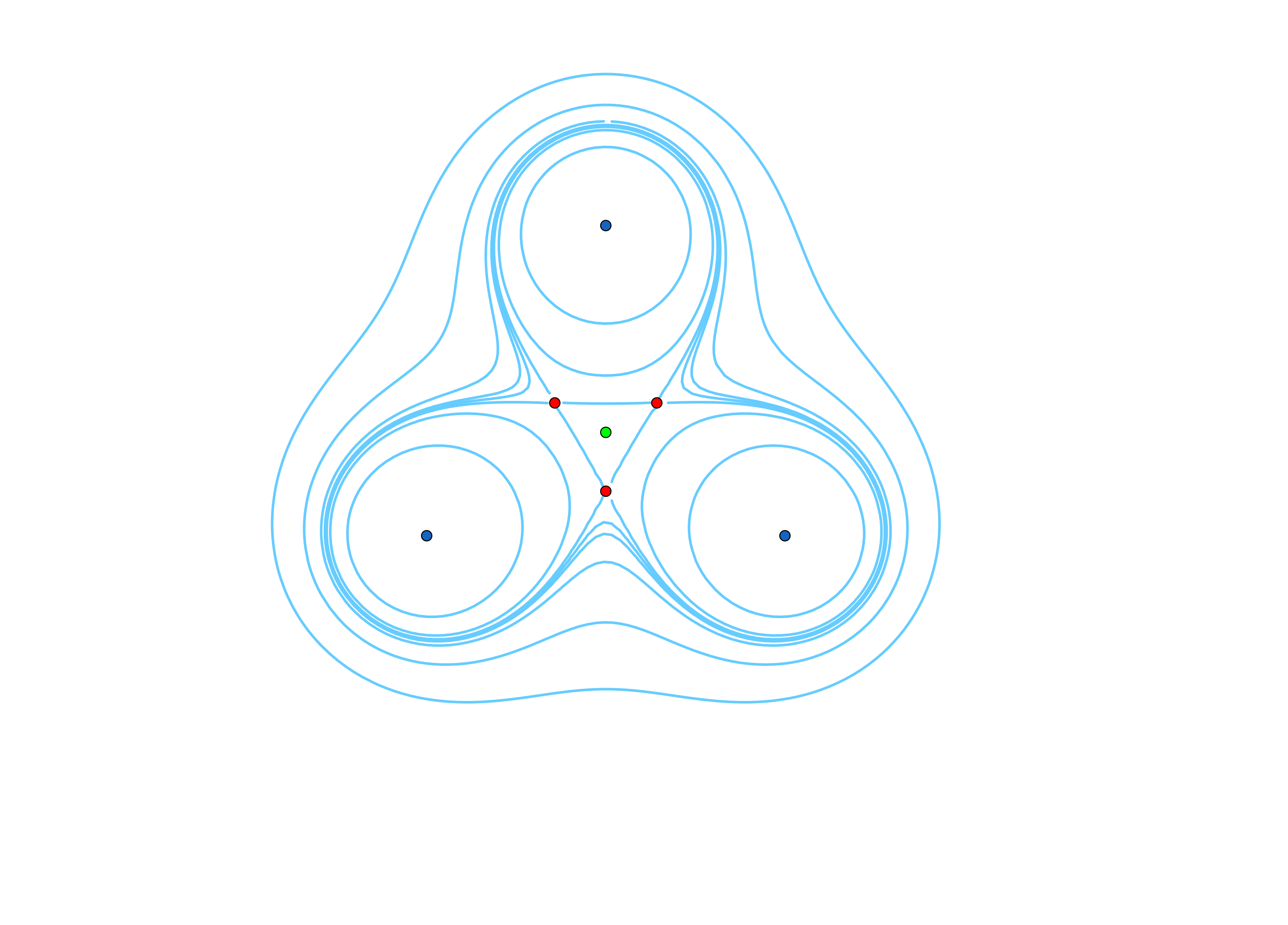}
			\caption{The points $\lbrace x_1, x_2 ,x_3 \rbrace$, in dark blue, are positioned at the vertices of an equilateral triangle and correspond to the location of the charges generating the potentials $\phi_2=\phi_3$ which determine the hyperK\"ahler metric in $N^2$ and $N^3$. Some of the level sets of this potential are represented in the picture in light blue.\\
			One can check that there are $4$ critical points of $\phi_2=\phi_3$ inside the triangle and the circle fibers over these points in $N^2 \cong N^3$ correspond to closed geodesics on these hyperK\"ahler manifolds. Overall, the red points correspond to index one geodesics and that the green one to an index two geodesic.}
		\end{figure}

		Using our analysis we find that, together, $N^2_\epsilon$ and $N^3_\epsilon$ carry a total of $6=2 \times 3$ closed geodesics of index one and $2=2 \times 1$ of index two, see the illustration of these geodesics given in figure \ref{fig:triangle}. These are located within Hausdorff distance $\epsilon \kappa(\epsilon)$ from those in $\epsilon^2 (N^2 \cup N^3)$ .
		
		Next, we consider $N^4$ which is an ALF gravitational instanton of type $A_8$ and can thus be determined using the Gibbons--Hawking ansatz with a potential $\phi_4$ obtained from $9$ unit charges in $\mathbb{R}^3$. We shall locate these at the vertices of three equilateral triangles whose centers we position at the vertices of a much larger equilateral triangle, see figure \ref{fig:tritriangle} below.
		
		\begin{figure}[htp]\label{fig:tritriangle}
		\centering
		\includegraphics[width=0.45\textwidth,height=0.25\textheight]{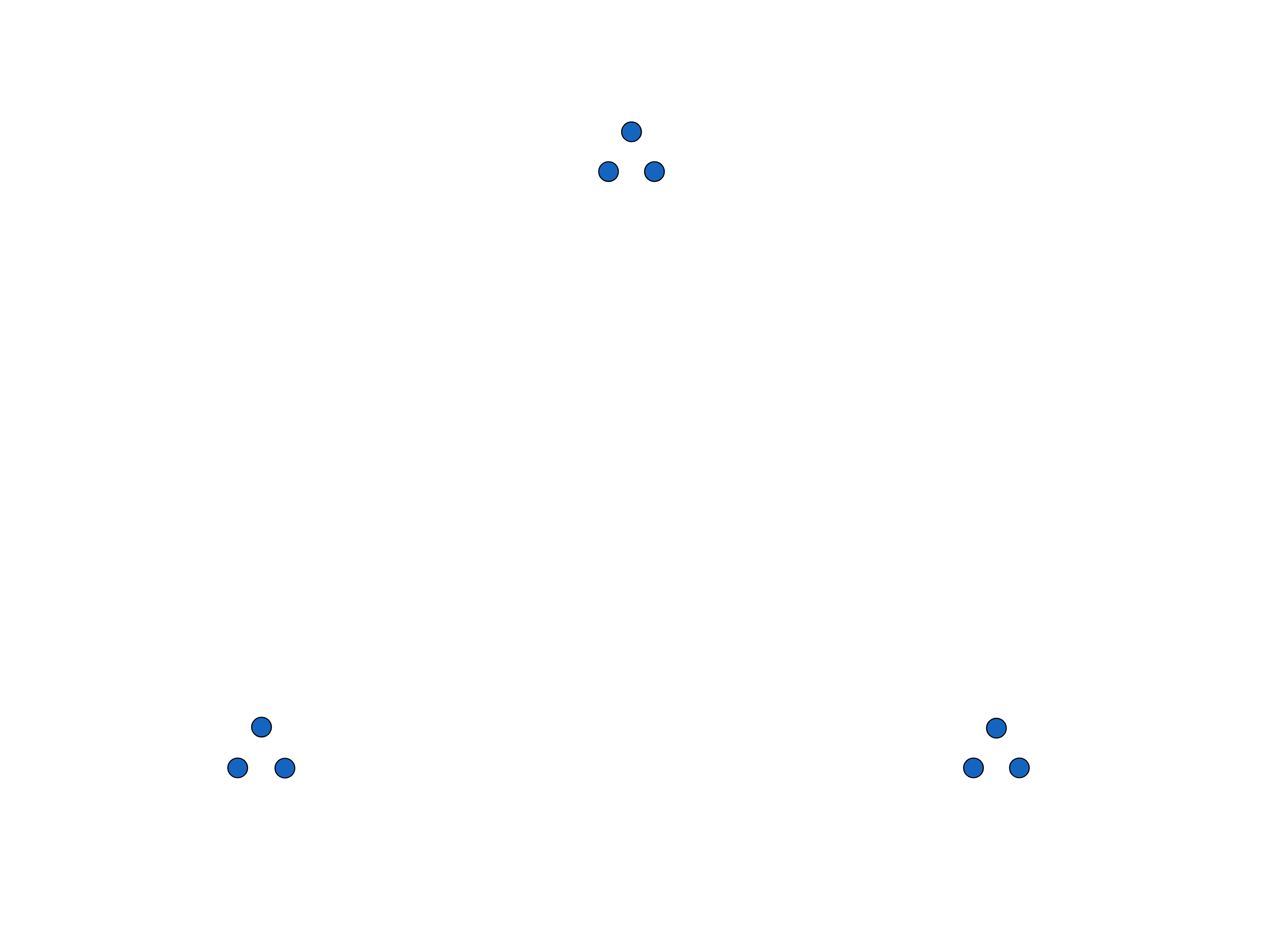}
		\caption{The $9$ points are located at the vertices of $3$ small equilateral triangles of side $d$ which we position at the vertices of a larger equilateral triangle of side $D \gg d$. As in \ref{fig:triangle}, there are $4$ critical points of $\phi_4$ inside each triangle, including the very large one. This results in a total of $12$ index one and $3$ index two closed geodesics.}
		\end{figure}
		 
		Overall, in the region $\bigcup_{i=1}^4 N^i_\epsilon$, we obtain $12+3 \times 2=18$ closed geodesics of index one and $4+2\times 1=6$ closed geodesics of index two. These, put together with those from before give a total of $23$ closed geodesics of index one and $11$ of index two in $K_\epsilon \cup \bigcup_{i=1}^4 N^i_\epsilon$.
		
		Furthermore, one can show that there are $8$ extra closed geodesics located in the set $\bigcup_{j=1}^8 M^i_\epsilon$, however their index is unknown. These come from deforming closed geodesics on the Atiyah--Hitchin manifold whose existence and moduli is known.
		
		This gives a total of $23+11+8=42$ closed geodesics whose positions we can approximately determine.
	\end{example}
	
	Let us see one more example.
	
	\begin{example}
		Again, we consider the situation when all $m_1 = \ldots = m_8=0$, but now with the partition of $16$ given by $4+4+4+4$ yielding $k_1=k_2=k_3=k_4=4$ and $n=4$. Then, we choose the points $\lbrace p_1, p_2, p_3, p_4 \rbrace$ so that, together with their symmetric image, they are the light blue points in figure \ref{fig:cubeG}. In this situation, we are able to locate $11$ closed geodesics on $K_\epsilon$ with precision $\epsilon^{13/10}$ as illustrated in the same figure. This is only one more than the lower bound given by Theorem \ref{thm:Main}. However, in this example we can determine the location of the geodesics simply using symmetry principles instead of having to find the critical points of $h$ which would require the use of computer approximations.
		
		\begin{figure}[htp]\label{fig:cubeG}
			\centering
			\includegraphics[width=0.65\textwidth,height=0.35\textheight]{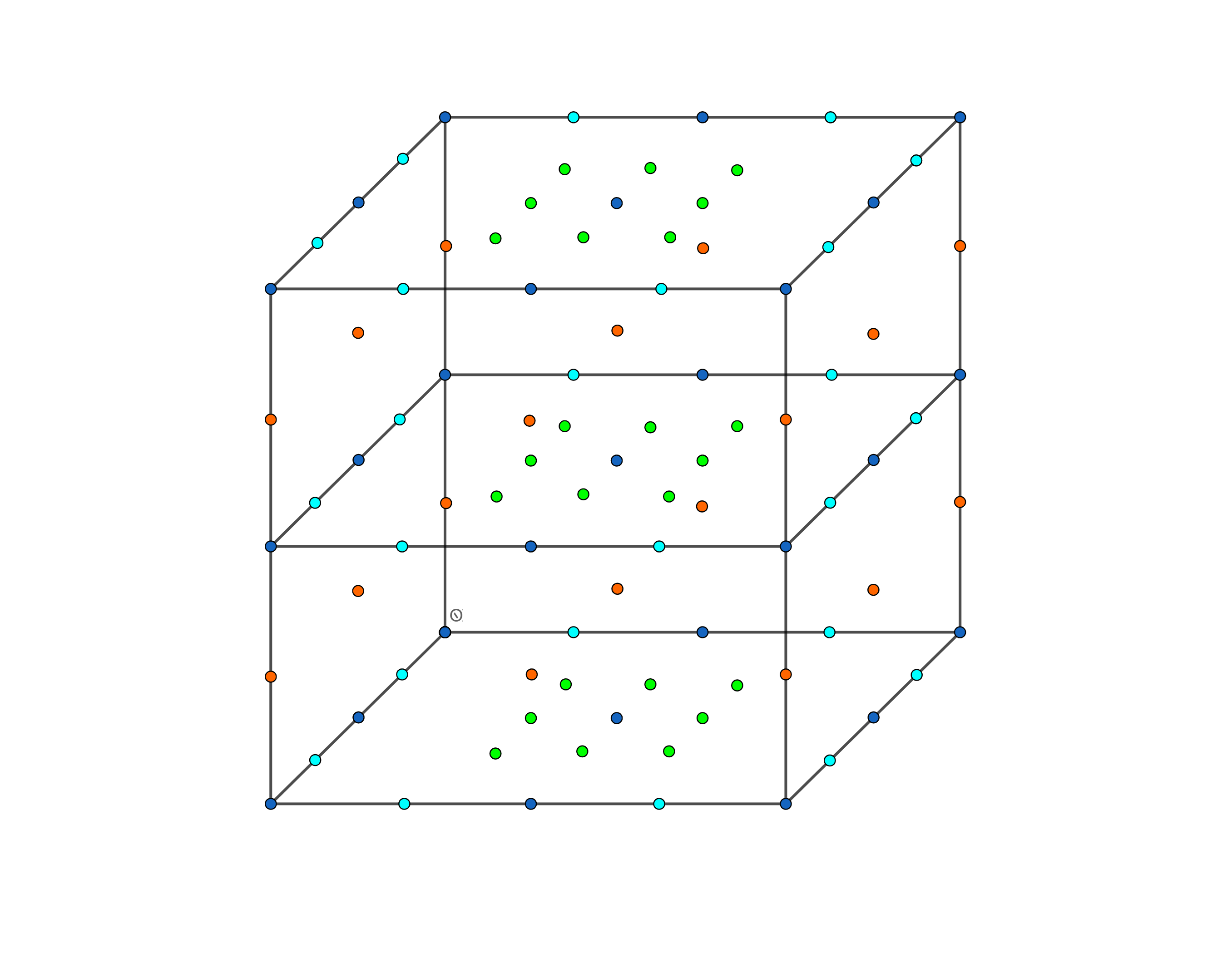}
			\caption{The points $\lbrace p_1, p_2, p_3, p_4 , -p_1,-p_2,-p_3,-p_4\rbrace$ are represented in light blue while those of $\lbrace q_1 , \ldots , q_8 \rbrace$ in dark blue. The function $h$ determining the metric in $K_\epsilon^{gh}$ is the electric potential generated by charges with value $(-4)$ at the dark blue points and charges with value $4$ at the light blue points. Symmetry arguments carried out in section \ref{ss:cube} show that $h$ has critical points at the orange and green points. This results in a total of $11=22/2$ closed geodesics on $K_\epsilon$.}
		\end{figure}
	
		As for the $N^i$, these are ALF gravitational instantons of type $A_3$. We choose all these to be isometric and given by the Gibbons--Hawking ansatz with potential $\phi$ having its singularities placed at the vertices of a square. Then, in each of these, there are $4$ closed geodesics of index one and $1$ of index two.
		
		In total, we obtain $11+4\times 5=31$ closed geodesics.
	\end{example}

	\subsection{Other results}
	
	Along the way towards proving our main goal we also establish other results which are interesting in their own right. For the reader's convenience, we state them here in passing.
	
	\begin{theorem}[No stable geodesics on some non-compact hyperK\"ahler $4$-manifolds]\label{thm:no stable closed geodesics}
		There are no closed geodesics stable to second order on Atiyah--Hitchin, Taub--Nut, and ALE or ALF $A_1$ gravitational instantons.
	\end{theorem}

	There are other interesting non-existence results in the literature. See \cite{Olsen2} which considers the Eguchi--Hanson space and \cite{Tsai2} for examples on other special holonomy Riemannian manifolds. 
	
	The next result constructs closed geodesics on hyperK\"ahler $4$-manifolds which are stable to second order, but which are overall unstable, i.e. not local minima of the length functional.

	\begin{theorem}[Unstable geodesics which are stable to second order]\label{thm:stable to second order INTRO}
		For infinitely many $k \in \mathbb{N}$, there is an ALE gravitational instanton of type $A_k$ containing a closed geodesic which is unstable, but is stable to second order.
	\end{theorem}

	Finally, by feeding Theorem \ref{thm:no stable closed geodesics} into Foscolo's construction one obtains the following result.

	\begin{theorem}[Stable geodesics forbidden to enter ``bubble'' region]
		Let $X_\epsilon$ be a Foscolo K3 surface constructed with all $m_1=\ldots =m_8=0$ and $k_1, \ldots , k_n \in \lbrace 1, 2 \rbrace$. Then, any closed geodesic in $X_\epsilon$ which is stable to second order must be totally contained in $K_\epsilon$, i.e. it does not enter the bubble regions $\bigcup_{j=1}^8 M^j_\epsilon \cup \bigcup_{i=1}^n N^i_\epsilon$.
	\end{theorem}

	\subsection{Organization}
	
	This article is organized in the following manner. We start in section \ref{sec:Preliminaries} by reviewing what the existing literature on closed geodesics has to say about this problem. Namely, we recapitulate what Morse theory for the energy functional implies for the present situation and by considering Kummer's K3 surfaces in passing. Section \ref{sec:Foscolo K3} is rather short and is devoted to explaining Foscolo's construction of Ricci-flat K\"ahler metrics on K3 surfaces with an eye towards our applications. 
	
	In sections \ref{sec:D0} and \ref{sec:Gibbons-Hawking} we consider the problem of constructing closed geodesics on the building blocks of Foscolo's K3 surfaces. These sections contain a mixture of original results with others that are drawn from the literature. Furthermore, in section \ref{sec:Jacobi} we investigate some geometric/analytic properties of these geodesics on the building blocks and abstract them into propositions which will be useful when it comes to deform these geodesics later on.
	
	Finally, in sections \ref{sec:deform geodesics concentrated} and \ref{sec:deform geodesics collapse} we deform the geodesics on the building blocks to actual geodesics on Foscolo's K3 surfaces. These sections respectively focus on the curvature concentrated and collapsing regions. We finalize the article with section \ref{sec:examples} which explores several examples and serve as applications of our results.

	\subsection{Acknowledgments}
	
	The author is very grateful for many discussions which highly improved his understanding of the subject of this article. In particular, for this reason, the author would like to thank Arnav Tripathy, Jørgen Olsen Lye, Lorenzo Foscolo, Max Zimet, Ruobing Zhang, Song Sun, Tamás Hausel, Willem Salm, Yu-Shen Lin.
	
	This work was in part motivated by a talk that the author saw \cite{DouglasTalk}, where Michael Douglas showed a numerical approximation of a closed geodesic in a Calabi--Yau manifold.  
	
	This work was supported by the NOMIS foundation as it was started while the author was a NOMIS Fellow in the Hausel group at IST Austria and finished at the author's current institution Instituto Superior T\'ecnico. The author also had a research membership at SLMath (former MSRI) whose hospitality he wants to acknowledge. The work here presented was developed at all these three institutions.

	\section{Prelimaries}\label{sec:Preliminaries}
	
	In this section we lay out some preliminaries which are supposed to answer some basic questions that may arise in the reading of the rest of this article. The contents of this section are as follows. In \ref{ss:length-energy-nullity-index} we recall the definitions of length, energy, nullity and index of a geodesic. In doing so, we also must recapitulate the definition of the Jacobi operator. In \ref{ss:stability} we lay out $4$ different notions of stability for a closed geodesic in a Riemannian manifold and explain how they are related to one another. 
	
	The subsections \ref{ss:Morse Theory} and \ref{ss:Kummer} already contain somewhat original results, but they result from simple applications of known theory. In \ref{ss:Morse Theory} we explain what basic Morse theory arguments have to say about the existence of index one closed geodesics on K3 surfaces. Finally, in \ref{ss:Kummer} we use White's deformation theorem in order to locate some closed geodesics for the Ricci-flat K\"ahler metrics on the K3 surfaces obtained from Kummer's construction.

	\subsection{Length, energy, nullity and index}\label{ss:length-energy-nullity-index}
	
	Let $\mathcal{I}=[0,2\pi]$ or $\mathcal{I}=S^1$. The length of a path $\gamma:\mathcal{I} \to X$ in a Riemannian manifold $(X,g)$ is defined as
	$$\ell(\gamma)=\int_\mathcal{I} |\dot{\gamma}(t)| dt ,$$
	where $|\dot{\gamma}(t)|^2 = g (\dot{\gamma}(t),\dot{\gamma}(t))$. In many applications, one is interested in finding paths which minimize this functional (for $\mathcal{I}=[0,2\pi]$ we must have both $\gamma(0)$ and $\gamma(2\pi)$ fixed). These are in particular geodesics. However, given its better analytic properties, it is more convenient to instead minimize the energy functional
	$$\cE(\gamma)=\int_\mathcal{I} |\dot{\gamma}(t)|^2 dt$$
	which, due to the Sobolev embedding $W^{1,2}(\mathcal{I}) \hookrightarrow C^{0,\frac{1}{2}}(\mathcal{I})$ and the Arzel\'a-Ascoli theorem, satisfies the Palais--Smale condition.
	
	Infinitesimal variations of $\gamma$ can be parametrized using normal vector fields $V \in \Gamma(\mathcal{I}, \gamma^*(TX)^\perp)$, i.e. those vector fields so that $g(\dot{\gamma}, \gamma_* V)=0$ pointwise. Omitting push-forwards and pull-backs, such a vector field gives rise to a variation of the form  
	$$f(t,s):= \exp_{\gamma(t)} \left( sV(t) \right).$$
	Then, using the identity
	\begin{align*}
		\frac{D}{ds} \frac{\partial f}{\partial t}  = \frac{D}{dt} \frac{\partial f}{\partial s},
	\end{align*}
	we consider $\cE(f(t,s))$, $\ell(f(t,s))$ and we find that their first variations, at $s=0$, are given by
		\begin{align*}
		(\partial_{s} \cE)|_{s=0}  & =  - 2 \int_\mathcal{I} g \left( \frac{D^2 \gamma}{dt} , V \right) dt ,  \\
		(\partial_{s} \ell)|_{s=0} & =  - \int_\mathcal{I} g \left( \frac{D}{dt} \left( \frac{\dot{\gamma}}{|\dot{\gamma}|} \right), V \right) dt ,
	\end{align*}
	the later vanishes if $\gamma$ is a geodesic while the first only vanishes if $\gamma$ is further parametrized using a multiple of the arc-length parameter. On the other hand, using the identity
	\begin{align*}
		\frac{D}{ds} \frac{D}{dt} \frac{\partial f}{\partial s}  = \frac{D}{dt} \frac{D}{d s} \frac{\partial f}{\partial s} + \mathrm{Riem} \left( \frac{\partial f}{\partial s} , \frac{\partial f}{\partial t} \right) \frac{\partial f}{\partial s} ,
	\end{align*}
	we find that the second variations are given by
	\begin{align*}
		(\partial^2_{s} \cE)|_{s=0} & = - 2 \int_\mathcal{I}  g \left( \frac{D^2 \gamma}{dt} , \frac{D^2 f}{ds^2} \Big\vert_{s=0}  \right) dt  + 2 \int_0^{2\pi}  g \left(  - \frac{D^2 V}{dt^2} + \mathrm{Riem}( \dot{\gamma}  , V) \dot{\gamma} ,V \right)  dt
	\end{align*}
	and
	\begin{align*}
		(\partial^2_{s} \ell ) |_{s=0} & =  - \int_\mathcal{I} g \left( \frac{D}{dt} \left( \frac{\dot{\gamma}}{|\dot{\gamma}|} \right), \frac{D^2 f}{ds^2} \Big\vert_{s=0} \right) dt + \int_0^{2\pi} g \left(  - \frac{D^2}{dt ds} \left( \frac{\dot{\gamma}}{|\dot{\gamma}|} \right) + \mathrm{Riem}(\dot{\gamma},V) \frac{\dot{\gamma}}{|\dot{\gamma}|}  , V \right)  dt .
	\end{align*}
	Again, the first of these terms vanish if $\gamma$ is a geodesic. Furthermore, if $\gamma$ is parameterized with respect to a constant multiple of the arc-length parameter, then the second terms in these second variations become the same (up to an overall constant). Motivated by this, one defines the Jacobi operator 
	$$J_{\gamma}: \Gamma(\mathcal{I}, \gamma^*(TX)^\perp) \to \Gamma(\mathcal{I}, \gamma^*(TX)^\perp),$$
	of a curve $\gamma$ as the negative of the operator associated with the second variation (when evaluated at a geodesic), i.e.
	\begin{equation}
		J_{\gamma}(V) = \frac{D^2 V}{dt^2} - \mathrm{Riem}( \dot{\gamma}  , V) \dot{\gamma} .
	\end{equation}
	Notice in particular that $(\partial_s^2 \cE)|_{s=0} = - \int_\mathcal{I} g (V, J_\gamma (V)) dt$. Then, one defined the nullity of a geodesic $\gamma$ as $\dim \left( \ker (J_{\gamma}) \right)$, and its index of $\gamma$ as the dimension of the maximal subspace on which $-J_{\gamma}$ is negative definite.

	\subsection{Notions of stability}\label{ss:stability}
	
	At this point it is convenient to consider several different notions of stability for geodesics. A geodesic is said to be (strictly) stable if it is a (strict) local minimum of $\ell$, or equivalently $\cE$. On the other hand, a geodesic is called stable to second order if it had index zero, and strictly stable to second order if furthermore it has
	vanishing nullity. These different notions of stability are related in the following manner: strictly stable geodesics are immediately stable, while those which are strictly stable to second order are in particular stable to second order. Furthermore, any stable geodesic is in particular stable to second order.
	
	In the particular case when $(X^4,g)$ is a $4$-dimensional hyperK\"ahler manifold, as is the case that we consider in this article, a result of Bourguignon--Yau \cite{BY} using the variations $I \dot{\gamma}$, $J\dot{\gamma}$ and $K\dot{\gamma}$ together with $\mathrm{Ric}=0$ shows that along any closed geodesic which is stable to second order, the full curvature vanishes, i.e. $\mathrm{Riem}=0$ along $\gamma(\mathcal{I})$. Given that $I,J,K$ and $\dot{\gamma}$ are parallel, it then follows that all these three variations are in the kernel of the Jacobi operator and so any such geodesic actually has nullity at least three and so cannot be strictly stable to second order. To the author's knowledge, no stable closed geodesic on an hyperK\"ahler 4-manifold is known to exist. Despite this fact, in \cite{Douglas}, two conjectures are posed on the existence of stable closed geodesics on (compact) Calabi-Yau manifolds. These posit that on any such there are several stable closed geodesics.
	
	As we shall later see in example \ref{ex:second order stable geodesic} by an explicit example, it is possible to have a closed geodesic in such an hyperK\"ahler manifold that is not stable, but is stable to second order.
	
	We further remark on the existence of a closed geodesic, stable to second order, on a specific K3 surface constructed using the Kummer construction. This was established in \cite{Olsen}, see also \cite{Olsen2}. However, whether this geodesic actually correspond to a local minima of the length functional is not known.

	\subsection{Geodesics from Morse theory}\label{ss:Morse Theory}
	
	\subsubsection{Existence of index one geodesics}
	
	Let $LX$ denote the Loop space of $X$, there is an isomorphism between $\pi_1(LX)$ and $\pi_2(X)$. Furthermore, if $X$ is simply connected, as is the case when $X$ is the differentiable manifold underlying the K3 surface, then by the Hurewicz isomorphism $\pi_2(X) \cong H_2(X, \mathbb{Z})$. Combining these two isomorphisms we conclude that
	$$b_1(LX)=b_2(X).$$
	Given that geodesics are critical points of the energy functional, $\cE$, which is Palais-Smale in $C^{k,\alpha} (S^1,X)$, we can apply Morse theory for $\cE$ and conclude from the Morse inequalities, as in Theorem 2.4.12 of \cite{Klingenberg}, that for the generic metric on $X$ the number of index one closed geodesics is at least $b_2(X)$ for the generic metric on $X$.\footnote{As stated in \cite{Klingenberg} the necessary condition on the metric is that all simple closed geodesics be non-degenerate, i.e. have zero nullity. However, by White's bumpy metrics theorem, this is the case for the generic metric on $X$.}
	
	For the K3 surface, we have $b_2(X)=22$ but it is not clear that the Ricci-flat K\"ahler metrics on it are generic in the intended sense, in fact it seems very likeley that they are not. For instance, it follows from the discussion in \ref{ss:stability} that if a closed geodesic stable to second order exists, then it has nullity at least three and so the metric cannot be generic in the needed sense. Instead, one may use the classes in $H_1(LX)$ to carry out a min-max argument and construct $22$ closed geodesics on $X$. However, this argument does not immediately imply that these have index one, just that the index is at least one.

	\subsubsection{Infinitely many closed geodesics}
	
	It is a well known result of Gromoll and Meyer, \cite{Gromoll}, that if $H^\ast (LX)$ is infinitely generated, then there are infinitely many geometrically distinct closed geodesics on $X$ which can be associated with the classes generating $H^\ast (LX)$. The index of these geodesics coincides with the degree of the class in $H^\ast (LX)$ which gives rise to it through the Morse theory. Furthermore, by \cite{VS} and \cite{Poirrer}, this is a the case for a very large class of differentiable manifolds, including the one underlying K3 surface. Hence, Morse theory for the energy functional guarantees the existence of infinitely many closed geodesics on any hyperk\"ahler K3 surface. However, the index of these geodesics increases in an unbounded manner with the degree of the generator of $H^*(LX)$ which gives rise to it. Furthermore, from abstract Morse theoretic arguments such as this one, nothing is learned about the location of the resulting geodesics and their other geometric properties.
	
	\begin{remark}
		We leave it as a suggestion for future work, the problem of computing a minimal model for $LX$ using the method of \cite{VS} and use it to more fully investigate the Morse theory for $\cE$. We believe this will lead to interesting results which one can try and relate to the Laplace spectrum by making use of the trace formula as in \cite{Guillemin_Weinstein,Zelditch}.
	\end{remark}

	\subsection{Closed geodesics on Kummer's K3 surfaces}\label{ss:Kummer}
	

	In this section we explain how to locate $48$ closed geodesics on the K3 surfaces constructed using the Kummer construction, which we shall refer to as Kummer's K3 surfaces. Recall that these are obtained by starting with a flat $4$-torus $\mathbb{T}_\Lambda=\mathbb{C}^2/\Lambda$ where $\Lambda$ is a full rank lattice and the orbifold $T_\Lambda/\mathbb{Z}_2$ where $\mathbb{Z}_2$ acts via $(z_1,z_2) \mapsto (-z_1,-z_2)$. Locally, the 16 singularities are modelled on $\mathbb{C}^2/\mathbb{Z}_2$ and can be desingularized by blwoing up. This gives a complex surface $X$ which is a K3 surface. 
	
	For $\epsilon=(\epsilon_1, \ldots , \epsilon_{16})$, with $|\epsilon| \ll 1$, one then considers an approximately Ricci-flat metric $g_{\epsilon}$ obtained by gluing the flat metric on the complement of the singular points in $\mathbb{T}_\Lambda/\mathbb{Z}_2$ with $16$ copies of the Eguchi-Hanson metric around each exceptional divisor $E_i$ having volume $\epsilon_i^2$. These metrics can then be perturbed to obtain an hyperk\"ahler metric $g_\epsilon^{hk}$ on $X$ which as $|\epsilon| \to 0$ converges to the flat orbifold $\mathbb{T}_\Lambda/\mathbb{Z}_2$. The details of this construction can be found in \cite{LS,Donaldson,Kobayashi}. 
	
	Furthermore, fixing open neighborhoods $U_i$ of the exceptional divisors $E_i$, for $i=1, \ldots , 16$, the rescaled metrics $\epsilon_i^{-2}g_\epsilon^{hk}|_{U_i}$ geometrically converge, as $\epsilon_i \to 0$, to a fixed Eguchi-Hanson metric where the volume of the totally geodesic $2$-sphere is $1$. Again, for the details we refer the reader to one the previously mentioned references, namely \cite{LS,Donaldson,Kobayashi}. Also, recall that the totally geodesic $2$-sphere in the Eguchi-Hanson space is actually a round $2$-sphere, and so the space of closed geodesics on such a round $2$-sphere is $\mathbb{RP}^3$ which has Lusternik--Schnirelmann number $3$.
	
	Hence, it follows from applying Brian White's deformation, stated as Theorem 3.2 in \cite{White}, that for $\epsilon_i \ll 1$ there are at least $3$ closed geodesics for $g_\epsilon^{hk}$ inside $U_i$. This results in a total of $16 \times 3= 48$ closed geodesics on any Kummer K3-surface, all of which have index at least $1$. 
	
	Furthermore, by example \ref{ex:A_8} in \ref{sec:Gibbons-Hawking} and continuity\footnote{The conclusion of this example is not original. It is contained at least in the references \cite{Chen} and \cite{Olsen}}, we must have that the Riemann curvature tensor of $g_\epsilon^{hk}$ never vanishes inside each open set $U_i$ and so no stable geodesic in $(X,g_\epsilon^{hk})$  can penetrate the region $\bigcup_{i=1}^{16} U_i$. This was already observed in \cite{Olsen2} and contradicts a possibility raised in section 4.4. of \cite{Douglas}.

%
%
%

	\section{Foscolo's $K3$ surfaces}\label{sec:Foscolo K3}
	
	In this section we shall describe Foscolo's construction of Ricci-flat K\"ahler metrics on K3 surfaces. The original reference for this construction is \cite{Foscolo1}, and we also recommend the survey article \cite{Foscolo2}. From now on we shall $\Lambda$ be a full rank lattice in $\mathbb{R}^3$. Then, we let $(\mathbb{T}_\Lambda = \mathbb{R}^3 / \Lambda ,g_E)$ be a flat $3$-torus and $\tau: \mathbb{T}_\Lambda \to \mathbb{T}_\Lambda$ the involution $\tau(x)=-x$. The fixed points of this involution satisfy $2x =0 \mod \Lambda$ and consist of $8$ distinct points in $\mathbb{T}_\Lambda$ which we shall denote by $\lbrace q_1 , \ldots , q_8\rbrace$. Furthermore, we consider an extra set of $2n$ points $\lbrace p_1, \ldots , p_n , \tau(p_1) , \ldots , \tau(p_n) \rbrace$ and the punctured torus
	$$T:= \mathbb{T}_\Lambda \backslash \lbrace q_1 , \ldots , q_8, p_1, \ldots , p_n , \tau(p_1) , \ldots , \tau(p_n) \rbrace .$$  
	Then, fix $m_1, \ldots , m_8 \in \mathbb{N}_0$ and $k_1, \ldots , k_n \in \mathbb{N}$ satisfying 
	\begin{equation}\label{eq:necessary condition for Laplace equation}
		\sum_{i=1}^n k_i + \sum_{j=1}^8 m_j =16
	\end{equation}
	and solve the equation
	\begin{equation}\label{eq:Poisson equation}
		\Delta h = 2 \pi \sum_{i=1}^n k_i (\delta_{p_i} + \delta_{\tau(p_i)}) + 2 \pi \sum_{j=1}^8 (2m_j-4) \delta_{q_i}, 
	\end{equation}
	on $\mathbb{T}_\Lambda$. The condition \ref{eq:necessary condition for Laplace equation} implies that such an $h$ exists. Furthermore, the fact that $h$ is harmonic away from the punctures and the integrality of the periods of $ \frac{1}{2\pi}\ast_E d h$ implies that there is a circle bundle $\pi : P \to T$ equipped with a connection $\theta \in \Omega^1(T, \mathbb{R})$ such that
	$$d \theta = - \ast_E d h.$$
	Next, one considers the function $h_\epsilon=1+\epsilon h$ and the metric 
	\begin{equation}\label{eq:g_epsilon^{gh}}
		g_\epsilon^{gh}:= \epsilon^2 h_\epsilon^{-1} \theta^2 + h_\epsilon g_{E},
	\end{equation}
	which, by Lemma 4.9 in \cite{Foscolo1}, is well defined, as a symmetric $2$-tensor, in the total space of $P$, and positive definite away from the pre-image of small radius $\sim \epsilon$ around the points $q_1, \ldots , q_8$ and $p_1, \ldots , p_n, \tau(p_1), \ldots , \tau(p_n)$. The $\mathbb{Z}_2$-action defined by $\tau$ lifts to the total space of $P$ and leaves $g_\epsilon^{gh}$ invariant. Then, Foscolo constructs an approximate hyperk\"ahler metric $g_\epsilon$ on a manifold $X$ obtained from $P/\mathbb{Z}_2$ by gluing any initially fixed $D_{m_j}$ ALF gravitational-instantons to the points $q_j$ for $j=1, \ldots , 8$, and any ALF gravitational-instantons of type $A_{k_i-1}$ to the points $p_i$ for $i=1, \ldots , n$. This approximately hyperK\"ahler metric has the property that $g_\epsilon=g_\epsilon^{gh}$ in the open set $K_\epsilon$ defined as the complement of the pre-image in $P$ of radius $\sim \epsilon^{2/5}$ balls around the points $\lbrace q_1 , \ldots , q_8, p_1, \ldots , p_n , \tau(p_1) , \ldots , \tau(p_n) \rbrace$. 
	
	For sufficiently small $\epsilon \leq \epsilon_0$, these metrics are then perturbed to obtain a hyperk\"ahler metric $g_\epsilon^{hk}$. As a Riemannian manifold, $X_\epsilon=(X, g_\epsilon^{hk})$ can be decomposed in the following manner
	\begin{equation}\label{eq:decomposition of X epsilon}
		X_\epsilon = K_\epsilon \cup  \bigcup_{j=1}^8 M^j_\epsilon \cup \bigcup_{i=1}^n N^i_\epsilon.
	\end{equation} 
	Furthermore, these hyperk\"ahler metrics $g_\epsilon^{hk}$ on $X$ satisfies the following properties:
	\begin{enumerate}
		\item In the large open set $K_\epsilon$, $g_\epsilon^{hk}$ approximates $g_\epsilon$ at the following rate
		\begin{equation}\label{eq:metric estimate}
			\sup_{K_\epsilon} | g_\epsilon^{hk} - g_\epsilon |_{g_\epsilon} \leq \epsilon^{\frac{11-2\delta}{5}},
		\end{equation}
		for some fixed $\delta \in (-1/2,0)$ a-priori fixed. Furthermore, by elliptic regularity, see the discussion at the end of the proof of Theorem 6.15 in \cite{Foscolo1}, a similar bound holds for all derivatives of the metrics $g_\epsilon^{hk}$ and $g_\epsilon$.
		
		\item For each $j=1, \ldots , 8$, the rescaled metrics $\epsilon^{-2}g_\epsilon^{hk}$ restricted to $M_\epsilon^j$ uniformly converge, with as many derivatives as necessary, to the initially fixed ALF gravitational instanton of type $D_{m_j}$.
		
		\item For each $i=1, \ldots , n$, the rescaled metrics $\epsilon^{-2}g_\epsilon^{hk}$ restricted to $N_\epsilon^i$ uniformly converge, with as many derivatives as necessary, to the initially fixed ALF gravitational instanton of type $A_{k_i-1}$.
	\end{enumerate}

%

	\section{Atiyah--Hitchin manifold}\label{sec:D0}
	
	One of the building blocks for Foscolo's K3 surfaces are $D_m$ ALF gravitational-instantons. Therefore, in order to construct closed geodesics on Foscolo's K3 surfaces it is useful to understand the existence of these on such gravitational instantons. To the author's knowledge, this is largely unexplored ground.\footnote{It seems to the author that the techniques employed in this article may be used to construct closed geodesics on more general $D_m$ ALF gravitational-instantons by making use of a recent construction in \cite{SS}.} For this reason, in this section we shall simply restrict to the case of $D_0$ which consists of the Atiyah--Hitchin manifold.
	
	In \ref{ss:AH nonexistence} we use the Bourguignon--Yau result in order to prove that there are no stable closed geodesics on the Atiyah--Hitchin manifold. Then, in \ref{ss:AH existence}, we review the results of \cite{Montgomery} which construct a family of closed geodesics on the Atiyah--Hitchin manifold.

	\subsection{Non-existence of stable closed geodesics}\label{ss:AH nonexistence}
	
	As already mentioned in \ref{ss:stability}, it follows from a result of Bourguignon--Yau in \cite{BY}\footnote{See also \cite{Chen} which is very concise and straight to the point.} that along any closed geodesic which is stable to second order in a hyperK\"ahler $4$-manifold, the full Riemann curvature must vanish. Hence, in order to show that no such geodesics exist for a given hyperK\"ahler $4$-manifold, we need only show that the sectional curvatures never vanish at the same point. 
	
	Following \cite{AH} and \cite{Tsai}, fix an orthonormal coframing $\lbrace \sigma_1, \sigma_2, \sigma_3 \rbrace$ of $SU(2)$ which satisfies $d \sigma_i= \sigma_j \wedge \sigma_k$ where $(i,j,k)$ denotes a cyclic permutation of $(1,2,3)$. Then, the Atiyah--Hitchin metric is written 
	$$g=dr^2 + a(r)^2 \sigma_1^2 +b(r)^2 \sigma_2^2+ c(r)^2 \sigma_3^2,$$
	where $r \in [0,+\infty)$ and $a,b,c:[0,+\infty)\to \mathbb{R}$ satisfy the ODE's
	$$\frac{da}{dr}=\frac{a^2-(b-c)^2}{2bc},$$
	and cyclic permutations of $(a,b,c)$. From this form of the metric, we can use Cartan's structure equations in order to compute expressions for the sectional curvatures of $g$, see for example section 2.4. of \cite{Tsai}. One finds that three of these are given by $\frac{a''}{a},\frac{b''}{b},\frac{c''}{c}$. It is shown in Lemma 10.10 of \cite{AH} that $a''(r) \leq 0$ with equality attained only at $r=0$. Hence, any stable closed must be contained in $r=0$. This consists of a minimal $\mathbb{RP}^2$ on which $SO(3)$, the isometry group of $g$, acts transitively. In fact, $SO(3)$ acts transitively in the unit circle bundle in $T\Sigma$ and so in the space of geodesics starting in $\Sigma$, tangent to $\Sigma$.  Hence, if there was to exist any closed geodesic of $g$ totally contained in $\Sigma$, then this $\mathbb{RP}^2$ would be totally geodesic, which is not the case.
	
	\begin{remark}
		In fact, we have that in any compact subset of the Atiyah--Hitchin manifold $|\mathrm{Riem}|$ can be bounded from below by a positive constant. Now consider Foscolo's K3 surfaces $X_\epsilon$ constructed using some $D_0$ gravitational instanton as $M^j$. By continuity, for sufficiently small $\epsilon>0$, the Riemann curvature tensor does not vanish in $M^j_\epsilon$ and so, no putative stable closed geodesic in $X_\epsilon$ can enter $M^j_\epsilon$.
	\end{remark}

	\subsection{Existence of closed geodesics}\label{ss:AH existence}
	
	Closed geodesics for the Atiyah-Hitchin metric do exist. One family of such geodesics is constructed in \cite{Montgomery}. Recall that the isometry group of the Atiyah--Hitchin manifold is $SO(3)$ and so its action must preserve the space of closed geodesics. The geodesics described in \cite{Montgomery} correspond to $1$-parameter subgroups of $SO(3)$ and so the moduli space of geodesics on which they live is diffeomorphic to $SO(3)/SO(2) \cong \mathbb{RP}^2$. The strategy of that reference is to make such $SO(3)$-orbits of closed geodesics on the Atiyah-Hitchin manifold $\cM_2$ correspond to the zeros of a vector field in the reduced space $T^*\cM_2/SO(3)$. Then, the authors show the existence of one such zero. Furthermore, according to the discussion in the last page of \cite{Montgomery}, the linearization of this vector field at one of its zeroes is nondegenerate and so any closed geodesic corresponding to it must have nullity equal to $\dim  \mathbb{RP}^2 =2$. 
	
	Finally, we mention in passing that the index of these closed geodesics is unknown to the author, but it must be at least one from the discussion in \ref{ss:AH nonexistence}. 

	\begin{remark}
	The following comments are unrelated to rest of this article, but are included here for completeness. 
		\begin{itemize}
			\item See also \cite{Wojtkowski} where quasiperiodic geodesics on the Atiyah--Hitchin manifold are constructed.
			\item There exist also closed geodesics on moduli spaces of higher charge monopoles and these have been constructed by Bielawski \cite{Bielawski}.
		\end{itemize}
	\end{remark}

	\section{Gibbons-Hawking manifolds}\label{sec:Gibbons-Hawking}
	
	Another  building block in Foscolo's construction of Ricci-flat K\"ahler metrics on K3 surfaces are ALF gravitational-instantons of type $A_k$. Contrary to the case of $D_m$, a lot can be easily discovered about closed geodesics on these by making use of the Gibbons--Hawking ansatz to explicitly write down the corresponding hyperK\"ahler metrics. In fact, this will be one of our main sources of examples and it will be crucial to constructing geodesics on the regions $N^i_\epsilon$, for $i=1, \ldots , n$, but also in the large open set $K_\epsilon$.
	
	In \ref{ss:GH ansatz} and \ref{ss:geodeics GH} we review the Gibbons--Hawking ansatz and use this to construct closed geodesics on ALF gravitational instantons of type $A_k$. In \ref{ss:stable geodesics GH} we compute their Riemann curvature tensor. This is an extremely useful computation which will find several applications in our work. First, it is used in here to give some non-existence results for stable closed geodesics, and to give examples of closed geodesics which are stable to second order, but overall unstable. Secondly, it will be used later in Section \ref{sec:Jacobi} to compute the Jacobi operator of the geodesics we will find.

	\subsection{Gibbons--Hawking ansatz}\label{ss:GH ansatz}
	
	The Gibbons--Hawking ansatz is a beautiful way of expressing all complete hyperK\"ahler metrics with circle symmetry in an explicit way. This construction dates back to \cite{GH1,GH2} and we recommend those references for more details on this construction. For a short more complete summary than that reviewed here we also recommend \cite{Lotay}.
	
	Let $m \geq 0$, $k \in \mathbb{N}_0$, $\lbrace P_1 , \ldots , P_k \rbrace$ a discrete subset of $\mathbb{R}^3$ and
	$$\phi (x) = m + \sum_{i=1}^k \frac{1}{2|x-P_i|}.$$
	Recall that the metric in the Gibbons--Hawking ansatz given by
	\begin{equation}\label{eq:metric GH}
		g=\phi^{-1} \eta^2 + \phi g_E,
	\end{equation}
	where $g_E$ is the pullback of the Euclidean metric on $\mathbb{R}^3$ to $X$ via $\pi:X \to \mathbb{R}^3$, which is a circle bundle away from $\lbrace P_1 , \ldots , P_k \rbrace$ on which $\eta$ is a connection form satisfying 
	$$d\eta = - \ast_E d \phi.$$
	In fact, notice from this equation that $\phi$ determines both $X$ and $\eta$ and thus $g$. Indeed, $\Delta_{g_E}\phi= 2\pi \sum_{i=1}^k \delta_{P_i}$ and so $\frac{1}{2\pi}\ast_E d \phi$ is closed in $\mathbb{R}^3 \backslash \lbrace 0 \rbrace$ and has integral periods, thus being the curvature of a connection $\eta$ on a circle bundle $\hat{\pi}:\hat{X} \to \mathbb{R}^3 \backslash \lbrace 0 \rbrace$. Then $X$ is obtained from $\hat{X}$ by adding a point above each of the points $P_i$ with the local model around each of these points being the radially extended Hopf bundle. This yields $\pi:X \to \mathbb{R}^3$.

	\subsection{Existence and location of closed geodesics}\label{ss:geodeics GH}
	
	The existence of closed geodesics on gravitational instantons of type $A_k$ was investigated in \cite{Lotay} and we shall summarize here its main results.
	
	For the hyperk\"ahler metric $g=\phi^{-1} \theta^2 + \phi g_E$, the length of a geodesic orbit is given by the function $l(x)=\phi(x)^{-\frac{1}{2}}$, which vanishes at any points where the circle action collapses and we can regard as a function on $\mathbb{R}^3$. The first observation made in \cite{Lotay} is that geodesic orbits on $(M,g)$ are the pre-images of the critical points of $\phi$. It swiftly follows that if $\phi$ has at least two singularities then $\phi$ has at least one critical point and any such lies in the convex hull of the singularities of $\phi$. A particularly interesting and easy example is when all the $k$ singularities lie in a line. In this situation it is shown that there are precisely $k-1$ critical points all of which lying in a line between pairs of adjacent singular points. The key result regarding geodesic orbits proven in \cite{Lotay} is Theorem 3.6 in that reference which generalizes the example were all singularities are colinear. For completeness we shall recollect here that result.
	
	\begin{theorem}\label{thm:GH ALF A_k}
		Let $X$ be an hyperk\"ahler 4-manifold constructed using the Gibbons-Hawking anstaz with $\phi$ having $k\geq 2$ singularities. Then, for the generic arrangement of the singularities, there are at least $k-1$ geodesic orbits. Furthermore, there is $m_0 \geq 0$ such that for all $m \geq m_0$, these $k-1$ geodesics can be chosen to be of index $1$.
	\end{theorem}
	\begin{proof}
		The first part of this result follows from a Morse theory argument and we redirect the reader to the original reference \cite{Lotay} where it is proven that there are at least $k-1$ geodesics of index at least $1$. Here, we claim that $k-1$ of these have index exactly $1$. This follows from the fact that each such geodesic $\gamma=\pi^{-1}(p)$ corresponds to a critical point $p$ of $\phi$ and the Morse theoretic argument shows the existence of at least $k-1$ critical points of index $2$. Then, in Proposition \ref{prop:Jacobi eigenvalues} below we shall prove the existence of $m_0$ such that if $m \geq m_0$, then $\mathrm{index}(\gamma)=3-\mathrm{index} \Hess_p \phi$. Hence, in such a situation, any index $2$ critical point of $\phi$ corresponds to an index $1$ closed geodesic of $\phi$. 
	\end{proof}

	\begin{remark}\label{rem:rescaled statement 0}
	In view of remark \ref{rem:rescaled statement}, we can trade the hypothesis that $m \geq m_0$ in the statement of this result by the hypothesis that $|P_i-P_j| \geq d_0$ for some $d_0$ and all $i \neq j$.
	\end{remark}
	
	Notice also that the case when all the singularities of $\phi$ are colinear leads to the minimum number of allowed geodesic orbits as guaranteed by the previous result. It turns out that it is possible to find explicit examples whose number of geodesic orbits is greater than this minimum. The example provided in Proposition 3.7 of \cite{Lotay} which we shall restate here for convenience. See also figure \ref{fig:triangle} above and \ref{fig:triangle second} below which provide visual representations of these geodesics.
	
	\begin{proposition}\label{prop:triangle}
		Let $(X,g)$ be an hyperK\"ahler manifold constructed using the Gibbons--Hawking ansatz with $\phi$ having exactly $3$ singularities located at the vertices of an equilateral triangle. Then, for $m \geq m_0$, $g$ admits exactly $4$ geodesic orbits, one having index $2$ with the remaining three having index $1$.
	\end{proposition}
	\begin{proof}
		Again the proof is given in Proposition 3.7. \cite{Lotay} except for the statement that regarding the resulting index of these geodesics. The statement regarding the index follows from the fact that such geodesics correspond to critical points of $\phi$ with three of them having index $2$ and one index $2$. Again, by proposition \ref{prop:Jacobi eigenvalues} below, the index of the resulting geodesics gets reversed and we obtain the stated result.
	\end{proof}

	\begin{figure}[htp]\label{fig:triangle second}
		\centering
		\includegraphics[width=0.6\textwidth,height=0.35\textheight]{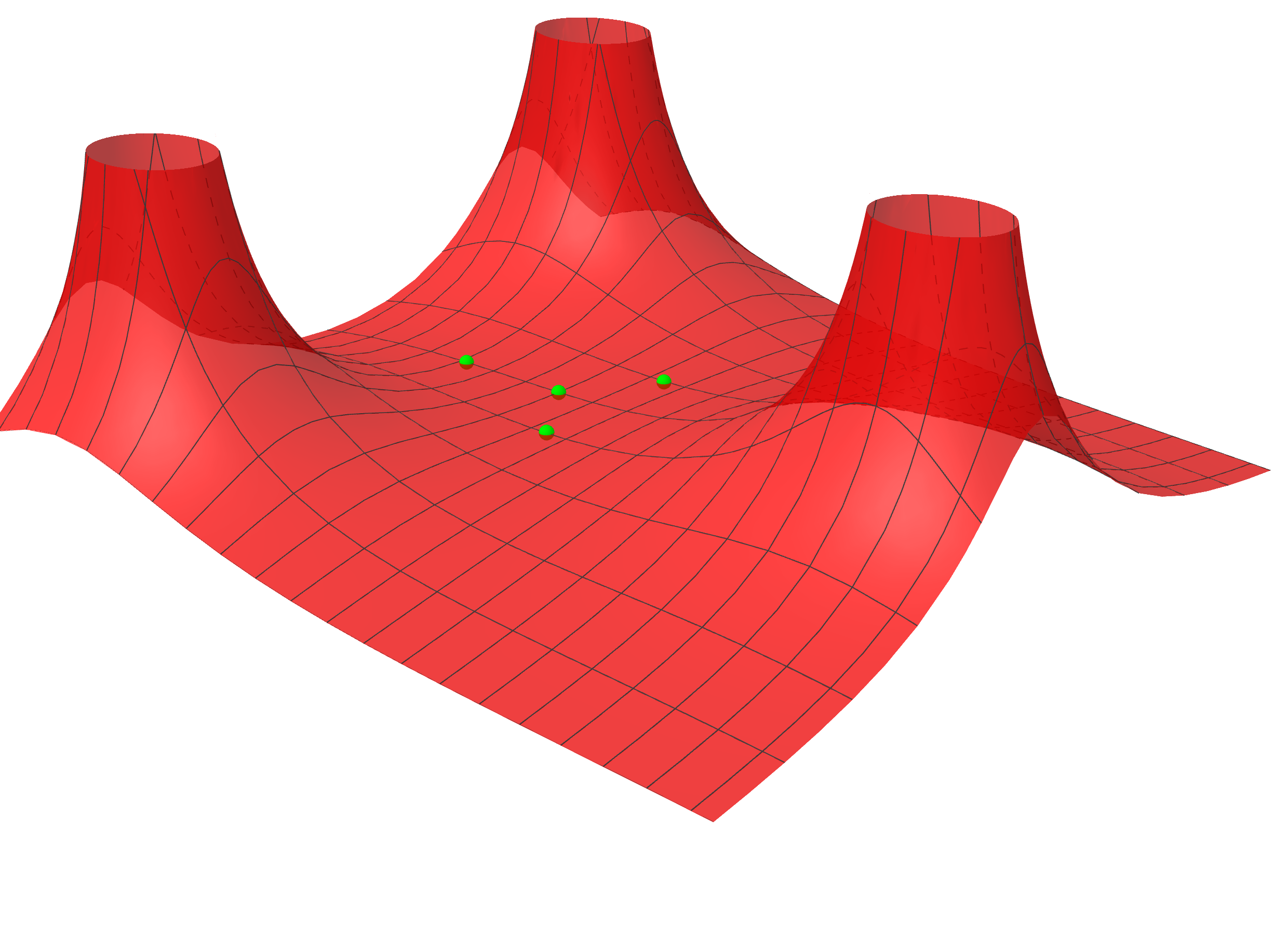}
		\caption{Plot of the potential $\phi$ restricted to the plane containing the triangle. The four critical points are inserted on top of the potential for easier visualization. They correspond to the geodesics from proposition \ref{prop:triangle} with the central one giving rise to the index two geodesic while the remaining three give rise to index one geodesics.}
	\end{figure}

	\subsection{Stable closed geodesics}\label{ss:stable geodesics GH}
	
	It was proven in \cite{Trinca} that there are no strictly stable geodesics on the hyperK\"ahler manifolds constructed using the Gibbons--Hawking ansatz. The argument is quite simple. By circle symmetry, any strictly stable geodesic would need to be circle-invariant, as otherwise it would have nonzero nullity. Hence, any such geodesic circle orbit would have to correspond to a local maximum of $\phi$ (minimum of $\ell$) which is impossible as $\phi$ is harmonic. 
	
	However, it turns out that it is possible to find infinitely many closed geodesics which are stable to second order, but of course not strictly stable. These have nullity equal to $3$, are unstable to higher order and we provide examples below, see for instance example \ref{ex:second order stable geodesic}.
	
	We start by computing the Riemann curvature tensor of the hyperk\"ahler metrics under consideration.

	\subsubsection*{Connection forms}
	
	Consider the orthonormal coframing $\lbrace e^{a} \rbrace_{a=0}^3$ given by
	\begin{equation}\label{eq: coframing GH R3}
		e^0=\phi^{-1/2} \eta\quad\text{and}\quad e^i=\phi^{1/2} d\mu_i
	\end{equation} 
	for $i=1,2,3$. Using Einstein's summation convention we have connection $1$-forms $\omega^a_b$ satisfying Cartan's structure equations
	\begin{equation}\label{eq:Cartan structure}
		d e^a = e^b \wedge \omega^a_b  , \ \ \  \omega^a_b + \omega^b_a =0.
	\end{equation}
	using the monopole equation 
	$$d\eta = -\frac{1}{2}\epsilon_{ijk} \frac{\partial\phi}{\partial \mu_i} d\mu_j\wedge d\mu_k,$$
	we deduce that
	\begin{align}\label{eq:connection_form_1}
		\omega^k_j & = \frac{1}{2\phi^{3/2}} \left( \frac{\partial \phi}{\partial \mu_j} e^k - \frac{\partial \phi}{\partial \mu_k} e^j - \epsilon_{ijk} \frac{\partial \phi}{\partial \mu_i} e^0  \right), \\ \label{eq:connection_form_2}
		\omega^0_i & =  \frac{1}{2\phi^{3/2}} \left( - \frac{\partial \phi}{\partial \mu_i} e^0 + \epsilon_{ijk} \frac{\partial \phi}{\partial \mu^j} e^k  \right).
	\end{align}

	\subsubsection*{Curvature Forms}
	
	Using the harmonicity of $\phi$ we may then compute the curvature forms 
	$$\Omega_{\mu}^\nu = d \omega_{\mu}^\nu + \omega_{\mu}^\kappa \wedge \omega_\kappa^{\nu}.$$
	This gives
	\begin{align}\label{eq:Omega 0i}
		\Omega_{0}^i & = - \frac{\phi_{ii}\phi - 2 \phi_i^2 +\phi_j^2+\phi_k^2 }{2\phi^3} (e_0 \wedge e_i - e_j \wedge e_k) + \\ \nonumber
		& \ \ \ \ - \frac{\phi_{ij}\phi - 3 \phi_i \phi_j}{2 \phi^3} (e_0 \wedge e_j - e_k\wedge e_i) - \frac{\phi_{ik}\phi - 3 \phi_i \phi_k}{2 \phi^3} (e_0 \wedge e_k - e_i\wedge e_j) 	\\ \label{eq:Omega jk}
		\Omega_{j}^k & =- \frac{\phi_{ii}\phi + \phi_j^2 + \phi_k^2 - 2 \phi_i^2}{2\phi^3} (e_0 \wedge e_i - e_j \wedge e_k) - \\ \nonumber
		& \ \ \ \ - \frac{\phi_{ij}\phi - 3 \phi_i \phi_j}{2\phi^3} (e_0 \wedge e_j - e_k \wedge e_i) - \frac{\phi_{ik}-3\phi_i \phi_k}{2 \phi^3} (e_0 \wedge e_k - e_i \wedge e_j),
	\end{align}
	where $(i,j,k)$ denotes a cyclic permutation of $(1,2,3)$ and in these equations \emph{no sum} is implicit over repeated indices. We therefore have
	$$|\mathrm{Riem}|^2 = \sum_{i=1}^3 2|\Omega_{0}^i|^2 + \sum_{i,j=1}^3 |\Omega_{i}^j|^2 , $$
	and using our formulas above gives
	\begin{align*}
		|\mathrm{Riem}|^2 & = \frac{1}{\phi^6}\sum_{i=1}^3 \left(  (\phi_{ii}\phi - 2 \phi_i^2 +\phi_j^2+\phi_k^2)^2 + (\phi_{ij}\phi - 3 \phi_i \phi_j)^2 + (\phi_{ik}\phi - 3 \phi_i \phi_k)^2  \right) + \\
		& \ \ \ \ + \frac{1}{\phi^6} \sum_{k=1}^3 \left( (\phi_{kk}\phi + \phi_i^2 + \phi_j^2 - 2 \phi_k^2)^2 + (\phi_{ki}\phi - 3 \phi_k \phi_i)^2 + (\phi_{kj}\phi -3\phi_k \phi_j)^2 \right) \\
		& = \frac{2}{\phi^6} \sum_{i=1}^3  \left(  (\phi_{ii}\phi - 2 \phi_i^2 +\phi_j^2+\phi_k^2)^2 + (\phi_{ij}\phi - 3 \phi_i \phi_j)^2 + (\phi_{ik}\phi - 3 \phi_i \phi_k)^2  \right) .
	\end{align*}
	Furthermore, use the splitting $\mathfrak{so}(4)=\mathfrak{su}(2)^+ \oplus \mathfrak{su}(2)^-$ we cab write the Riemann curvature tensor as
	\begin{eqnarray}\nonumber
		\mathrm{Riem} & = & \Omega_{\mu}^{\nu} e^\mu \otimes e^{\nu} \\ \nonumber
		& = & \sum_{i=1}^3\Omega_0^i e^0 \wedge e^i + \sum_{j <k=1}^3\Omega_j^k e^j \wedge  e^k \\ \nonumber
		& = & \sum_{i=1}^3\Omega_0^i \frac{\omega_i + \overline{\omega_i}}{2} + \sum_{j <k=1}^3\Omega_j^k \frac{\overline{\omega_i} - \omega_i}{2} 	\\ \nonumber
		& = & \underbrace{ \frac{1}{2} \sum_{i=1}^3 (\Omega^i_0 - \Omega^k_j) \otimes \omega_i }_{\mathrm{Riem}^+}  +  \underbrace{ \frac{1}{2} \sum_{i=1}^3 (\Omega^i_0 + \Omega^k_j) \otimes \overline{\omega_i}  }_{\mathrm{Riem}^-} ,
	\end{eqnarray}
	where $(i,j,k)$ is always a cyclic permutation of $(1,2,3)$. Then, using equations \ref{eq:Omega 0i} and \ref{eq:Omega jk}, we find that $\mathrm{Riem}^+$ vanishes, while $\mathrm{Riem} = \mathrm{Riem}^{-}$ with
	\begin{align*}
		\mathrm{Riem}^- & = - \sum_{i=1}^3 \left( \frac{\phi_{ii}\phi - 2 \phi_i^2 +\phi_j^2+\phi_k^2 }{2\phi^3} \ \overline{\omega_i}   + \frac{\phi_{ij}\phi - 3 \phi_i \phi_j}{2 \phi^3} \ \overline{\omega_j}  + \frac{\phi_{ik}\phi - 3 \phi_i \phi_k}{2 \phi^3} \ \overline{\omega_k} \right) \otimes \overline{\omega_i}  ,
	\end{align*}
	or in matrix form viewing $\mathrm{Riem}^-$ as an endomorphism of $\Lambda^2_-$.
	\begin{eqnarray}\label{eq:W matrix}
		-\mathrm{Riem}^- & = & \begin{pmatrix}
			\frac{\phi_{11}\phi - 2 \phi_1^2 +\phi_2^2+\phi_3^2 }{2\phi^3} & \frac{\phi_{12}\phi - 3 \phi_1 \phi_2}{2 \phi^3} & \frac{\phi_{13}\phi - 3 \phi_1 \phi_3}{2 \phi^3} \\
			\frac{\phi_{21}\phi - 3 \phi_2 \phi_1}{2 \phi^3} & \frac{\phi_{22}\phi - 2 \phi_2^2 +\phi_3^2+\phi_1^2 }{2\phi^3} &
			\frac{\phi_{23}\phi - 3 \phi_2 \phi_3}{2 \phi^3} \\ \frac{\phi_{31}\phi - 3 \phi_3 \phi_1}{2 \phi^3} &
			\frac{\phi_{32}\phi - 3 \phi_3 \phi_2}{2 \phi^3} &
			\frac{\phi_{33}\phi - 2 \phi_3^2 +\phi_1^2+\phi_2^2 }{2\phi^3}
		\end{pmatrix} ,
	\end{eqnarray}

	\begin{example}[Taub-Nut]
		In this case $\phi=m + \frac{1}{2 r}$ with $r=\sqrt{x^2+y^2+z^2}$. Inserting this into equation \ref{eq:W matrix} We compute 
		$$	-\mathrm{Riem}^- = \begin{pmatrix}
			\frac{2m (2x^2-y^2-z^2)}{r^2 (1+2mr)^3} & \frac{6m xy}{r^2 (1+2mr)^3} & \frac{6m xz}{r^2 (1+2mr)^3} \\
			\frac{6m yx}{r^2 (1+2mr)^3} & \frac{2m (2y^2-z^2-x^2)}{r^2 (1+2mr)^3} & \frac{6m yz}{r^2 (1+2mr)^3} \\
			\frac{6m zx}{r^2 (1+2mr)^3} & \frac{6m zy}{r^2 (1+2mr)^3} & \frac{2m (2z^2-x^2-y^2)}{r^2 (1+2mr)^3} 
		\end{pmatrix},$$
		and
		$$|	\mathrm{Riem}|^2=\frac{24m^2}{(2mr+1)^6}.$$
		which never vanishes and so there are no stable closed geodesics for the Taub--Nut metric.
	\end{example}

	\begin{example}[Orbifold Eguchi-Hanson]\label{ex:Eguchi--Hanson}
		Suppose the case when $\phi= \frac{k_1}{2r_1}+\frac{k_2}{2r_2}$ where $r_1=\sqrt{x^2+y^2+(z-1)^2}$, $r_2=\sqrt{x^2+y^2+(z+1)^2}$ and $k_1,k_2 \in \mathbb{N}$. We compute that 
		$$|	\mathrm{Riem} |^2= \frac{96k_1^2k_2^2}{(k_1r_1+k_2r_2)^6},$$ 
		and again we conclude that no stable closed geodesics exist.
		
		The nonexistence of stable closed geodesics on the Eguchi-Hanson space is proven in \cite{Olsen} and \cite{Chen}. The approach of those references is also based on \cite{BY}, but does not use the Gibbons--Hawking anstaz, relying instead in writing the metric in cohomogeneity-$1$ form.
	\end{example}
	
	\begin{example}[ALF version of Eguchi-Hanson]
		This is the main example we will be interested in and we find that $\phi= m+\frac{k_1}{2r_1}+\frac{k_2}{2r_2}$ where $r_1=\sqrt{x^2+y^2+(z-1)^2}$, $r_2=\sqrt{x^2+y^2+(z+1)^2}$ and $k_1,k_2 \in \mathbb{N}$. We find that
		$$\frac{\phi_{12}\phi - 3 \phi_1 \phi_2}{2 \phi^3}   = \frac{6xy}{r_1^2 r_2^2} \frac{mk_1 r_1^5 + m k_2 r_2^5 + 8z^2 k_1 k_2}{(2mr_1r_2 +k_1r_1+k_2r_2)^3},$$
		which is always non-negative unless $xy=0$. Given the symmetry, we can with no loss of generality simply analyze the situation when $x=0$. In that situation we find that along $x=0$
		$$\frac{\phi_{11}\phi - 2 \phi_1^2 +\phi_2^2+\phi_3^2 }{2\phi^3} = -2 \frac{mk_1r_1^3+mk_2r_2^3+2k_1k_2}{(2r_1r_2+k_1r_1+k_2r_2)^3},$$
		which can never vanish and we therefore find that no geodesic, stable to second-order, exist in this ALF version of Eguchi-Hanson.
	\end{example}
	
	\begin{example}[A non-minimal geodesic which is stable to second order]\label{ex:second order stable geodesic}
		Consider the points $p_1=(1,0,0)$, $p_2=(0,1,0)$, $p_3=(0,0,1)$ and the harmonic function
		$$\phi = \sum_{i=1}^3 \left( \frac{1}{2|x-p_i|}+ \frac{1}{2|x+p_i|} \right).$$
		One finds that $\partial_i\phi=0=\partial_{ij}^2\phi$ at the origin for all $i,j=1, \ldots , 3$. Hence, the Riemann curvature tensor of $g$ vanishes at the circle obtained as the pre-image of this point. We therefore find that this geodesic is stable to second order. However, it cannot be strictly stable as $\phi$ is harmonic and so it cannot have a local maximum at the origin by the strong maximum principle.
	\end{example}

	\begin{example}\label{ex:Platonic}
		Suppose the singularities of $\phi$ are placed at the vertices of a platonic solid. Then, the circle in $X$ above the origin is a closed geodesic which is stable to second order.
	\end{example}
	
	These examples can be generalized in a very interesting way. For this, we shall consider sufficiently symmetric configurations of singularities of $\phi$. Let $\mathcal{H}_2 \subset \mathbb{R}[x,y,z]$ denote the subspace generated by the harmonic homogeneous polynomial of degree $2$ and $G \subset O(3, \mathbb{R})$ denote the group of symmetries of the singularities of $\phi$. This group acts on $ \mathbb{R}[x,y,z]$ by precomposition and a computation shows that this action does preserve $\mathcal{H}_2$	. We shall denote by $\mathcal{H}_2^G \subset \mathcal{H}_2$ the subspace consisting of the $G$-invariant elements.
	
	\begin{theorem}
		Suppose that $\phi$ has a critical point at the origin and that the symmetry group $G$ of the configuration of its singularities is such that $\mathcal{H}_2^G = 0$. Then, the origin is a critical point of $\phi$ with vanishing Hessian and so the circle above it in $X$ corresponds to a closed geodesic which is stable to second order. However, it is not strongly stable.
	\end{theorem}
	\begin{proof}
		Given that $\phi$ has a critical point and thus no singularity at the origin it is real analytic in an open neighborhood of $0$ and so we can write it as a Taylor series
		$$\phi(x) = \sum_{m \geq 0} \phi_m(x) ,$$
		with each $\phi_m$ a harmonic homogeneous polynomial of degree $m$. Has $0$ is a critical point we must have $\phi_1=0$ and as $\phi_2 \in \mathcal{H}_2^G$ it also vanishes by hypothesis. Hence, as in Example \ref{ex:second order stable geodesic}, the curvature tensor of $g$ vanishes along the circle above the origin, thus proving that it corresponds to a geodesic stable to second order. Its instability follows from the fact that $\phi$ cannot have local minima and so there are nearby circle orbits $\pi^{-1}(x)$ with strictly smaller length $2\pi \phi(x)^{-1/2}$.
	\end{proof}

	Finally, as a corollary we prove Theorem \ref{thm:stable to second order INTRO} which we present here as a corollary.
	
	\begin{corollary}
		For infinitely many $k \in \mathbb{N}$, there are ALE and ALF gravitational instanton of type $A_k$ containing a closed geodesic which is unstable, but is stable to second order.
	\end{corollary}
	\begin{proof}
		Let $k+1$ admit a partition all of whose entries are either $3$, $6$, $8$, $12$ or $20$. Then, we can consider a family of concentric Platonic solids and place charges at their vertices. By examples \ref{ex:second order stable geodesic} and \ref{ex:Platonic} the resulting potential $\phi$ will have vanishing Hessian at their common center. Then, the corresponding ALE and ALF $A_{k}$ gravitational instantons constructed using the Gibbons--Hawking ansatz have a closed geodesic which is stable to second order but overall unstable. This is the pre-image of the center by the hyperK\"ahler moment map.
	\end{proof}

	\section{Jacobi operators in the building blocks}\label{sec:Jacobi}
	
	The goal of this section is to compute the index of geodesic orbits for the metrics constructed via the Gibbons--Hawking ansatz. These results are stated in Propositions \ref{prop:Jacobi eigenvalues} and \ref{prop:Jacobi eigenvalues 2} depending on whether the base is Euclidean space or a $3$-torus respectively. These results also characterize when such geodesic orbits are non-degenerate.
	
	Finally, in Proposition \ref{prop:Jacobi}, we invert the Jacobi operator of a non-degenerate geodesic orbit for $g_\epsilon^{hk}$, and give an estimate on the norm of its inverse.
	
	\subsection{Formula for the Jacobi operator of a geodesic orbit}
	
	In this initial subsection we shall compute the Jacobi operator of a geodesic circle orbit with respect to Gibbons--Hawking type metrics.
	
	\subsubsection{With base $\mathbb{R}^3$}
	
	For the metrics as in \ref{eq:metric GH} and using the co-framing \ref{eq: coframing GH R3}, we can write a normal vector $V$ to a geodesic orbit $\gamma_0=\pi^{-1}(p)$ as $V=\sum_{i=1}^3 V_i e_i$. Then, we compute
	\begin{align*}
		\frac{D}{dt} \frac{DV}{dt} & = \frac{D}{dt} ( \dot{V_i} e_i +  V_i \nabla_{e_0} e_i ) \\
		& =  \ddot{V_i} e_i + 2 \dot{V_i} \nabla_{e_0} e_i +  V_i \nabla_{e_0}^2 e_i  \\
		& =  \ddot{V_i} e_i ,
	\end{align*}
	and using the fact that at a critical point $p$ of $\phi$ we have $\Omega_0^j(e_0,e_i) = - \frac{1}{2\phi^2} \frac{\partial^2 \phi}{\partial \mu_i \partial \mu_j}$, we have
	\begin{align*}
		\textrm{Riem}(\dot{\gamma_0} , V) \dot{\gamma_0}  & = V_i \langle R(e_0,e_i) e_0 , e_j \rangle e_j \\
		& = V_i\Omega_0^j(e_0,e_i) e_j \\
		& = - \frac{1}{2\phi^2} \frac{\partial^2 \phi}{\partial \mu_i \partial \mu_j} V_i e_j .
	\end{align*}
	Putting these together, we find that
	\begin{align}\nonumber
		J_{\gamma_0}(V) & = \frac{D}{dt} \frac{DV}{dt} - \textrm{Riem}(\dot{\gamma_0} , V) \dot{\gamma_0} \\ \label{eq:Jacobi operator}
		& = \left( \ddot{V_j} + \frac{1}{2\phi^2} \frac{\partial^2 \phi}{\partial \mu_i \partial \mu_j} V_i \right) e_j .
	\end{align}
	
	\subsubsection{With base $\mathbb{T}_\Lambda$}
	
	We now turn to the metrics $g_\epsilon^{gh}$ defined in \ref{eq:g_epsilon^{gh}}, which are of the form $\epsilon$ times a metric constructed using the Gibbons--Hawking ansatz with base a $3$-torus $\mathbb{T}_\Lambda$. Given that we have computed the curvature tensor for a Gibbons--Hawking metric, it is convenient to work with
	$$\epsilon^{-1} g_\epsilon^{gh}= \frac{\epsilon}{1+\epsilon h} \theta^2 + \frac{1+\epsilon h}{\epsilon} g_E.$$ 
	Now, let $\lbrace e_0 , e_1, e_2, e_3 \rbrace$ be an orthonormal framing with respect to $\epsilon^{-1} g_\epsilon^{gh}$ which away from small balls around the singular points of $h$ is dual to the co-framing.
	$$e^0 = \sqrt{ \frac{\epsilon}{1+\epsilon h}} \theta , \ \text{and} \ e^i = \sqrt{ \frac{1+\epsilon h}{\epsilon}} dx_i ,  $$
	In particular, this gives $e_i = \sqrt{ \frac{\epsilon}{1+\epsilon h}} \partial_{x_i}$ and in equation \ref{eq:Omega 0i} we have computed the curvature forms of $\epsilon^{-1}g_\epsilon^{gh}$. We find that at a critical point of $\phi=\epsilon^{-1}+h$
	$$\Omega_0^j(e_0,e_i) = - \frac{1
	}{2\phi^2} \frac{\partial^2 \phi}{\partial \mu_i \partial \mu_j} = - \frac{\epsilon^2 V_i}{2(1+\epsilon h)^2} \frac{\partial^2 h}{\partial \mu_i \partial \mu_j} .$$

	At this point, we return to working with the metric $g_\epsilon^{gh}$. An orthonormal framing with respect to this metric is given by $E_i=\epsilon^{-\frac{1}{2}} e_i$. In particular, we have $E_0=\epsilon^{-1} \sqrt{1+\epsilon h} \theta^*$ with $\theta^*$ the dual to $\theta$ with respect to the basis $\theta, dx_1, dx_2, dx_3$. Also, we find that 
	$$\gamma_0^* g_\epsilon^{gh} = \frac{\epsilon^2}{1+\epsilon h} \gamma_0^*\theta^2=dt^2,$$ 
	where $t$ is the arclength parameter, which must therefore take values in $\left[0, \frac{\epsilon}{\sqrt{1+\epsilon h}} \right]$ and satisfy $\dot{\gamma_0}(t):= (d\gamma_0)(\partial_t)=E_0$. Then, writing a normal vector field $V$ is this basis as $V= \sum_{i=1}^3 V_i E_i$, we find that at a critical point of $h$
	\begin{align*}
		R(\dot{\gamma_0} , V) \dot{\gamma_0}  = V_i\Omega_0^j(E_0,E_i) E_j  = \epsilon^{-1} V_i \Omega_0^j(e_0,e_i) E_j  = - \frac{\epsilon V_i}{2(1+\epsilon h)^2} \frac{\partial^2 h}{\partial \mu_i \partial \mu_j} E_j ,
	\end{align*}
	where we have used the fact that the curvature $2$-forms are invariant by scaling. Next, it follows from the computation of the connection forms in \ref{eq:connection_form_2} that at a critical point of $\phi=\epsilon^{-1}+h$ we have $\nabla_{e_0} e_i=0= \nabla_{e_0}^2 e_i$ and so also $\nabla_{E_0} E_i=0= \nabla_{E_0}^2 E_i$, which results in $\frac{D^2V}{dt^2} =  \ddot{V_i} E_i$. Thus, putting these two computations together we obtain
	\begin{align}\nonumber
		J_{\gamma_0}^{gh}(V) & = \frac{D^2V}{dt^2} - \textrm{Riem}(\dot{\gamma_0} , V) \dot{\gamma_0} \\ \label{eq:Jacobi operator 2}
		& = \left( \ddot{V_j} + \frac{\epsilon V_i
		}{2(1+\epsilon h)^2} \frac{\partial^2 h}{\partial \mu_i \partial \mu_j} \right) E_j .
	\end{align}

	\subsection{Index of the Jacobi operator of a geodesic orbit}
	
	In this subsection we compute the index of the Jacobi operator associated with a geodesic orbit for the metrics given by the Gibbons--Hawking ansatz. For convenience, we shall again state the result separately for the cases when the base $\mathbb{R}^3$ and $\mathbb{T}^3$.
	
	\subsubsection{With base $\mathbb{R}^3$}
	
	Let $P_1, \ldots , P_K \in \mathbb{R}^3$ and $m \geq 0$ as in \ref{ss:GH ansatz}. In the next result, we shall consider the metric constructed using the Gibbons--Hawking ansatz with potential function 
	$$\phi=  m + \sum_{i=1}^k \frac{1}{2|x-P_i|}.$$
	
	In the next result, we shall relate the index of a geodesic of the form $\gamma_0=\pi^{-1}(p)$ for the hyperK\"ahler metric \ref{eq:metric GH}, with the index of the potential function $\phi$ at $p$.
	
	\begin{proposition}\label{prop:Jacobi eigenvalues}
		For $P_1, \ldots , P_K \in \mathbb{R}^3$ as above, there is $m_0 >0$ such that for all $m \geq m_0$, $p \in \mathbb{R}^3$ a critical point of $\phi$ and $V \in \mathbb{R}^3$ an eigenvector of $\mathrm{Hess}_p(\phi)$ associated with a non-negative eigenvalue $\mu$, the following holds. For $\gamma_0=\pi^{-1}(p)$, $V$ gives rise to a normal eigenvector-field of $-J_{\gamma_0}$ associated with the eigenvalue
		$$-\frac{\mu}{2\phi(p)^2} .$$
		Conversely, any non-positive eigenvector of $-J_{\gamma_0}$ arises in this way. In particular, we have
		$$\mathrm{index}\left(-J_{\gamma_0}\right) = 3- \mathrm{index} \Hess_p(\phi).$$
	\end{proposition}
	\begin{proof}
		In the situation of this proposition, $\gamma_0= \pi^{-1}(p)$ is a geodesic orbit of length $L=2\pi\phi(p)^{-1/2}$. Then, using an orthonormal framing $\lbrace e_0 , e_1, e_2, e_3 \rbrace$ and formula \ref{eq:Jacobi operator}, we have
		$$J_{\gamma_0}(V) = \left( \ddot{V_i} + \frac{1}{2 \phi^2} \frac{\partial^2 \phi}{\partial \mu_i \partial \mu_j} V_j \right) e_i. $$
		In particular, the operator $-J_{\gamma_0}$ is elliptic operator with a finite number of negative eigenvalues. We shall now look for non-positive eigenvalues $-\lambda \leq 0$ of $-J_\gamma$. For this, it is convenient to decompose any associated eigenvector $V(t)= \sum_{i=1}^3V_i(t) e_i$ into Fourier modes. Thus, writing its coefficients, $V_i(t)$ for $i=1,2,3$, in a cosine series, we have
		\begin{align*}
			V_i(t) = V_{i,0} + \sum_{n \in \mathbb{N}} V_{i,n} \cos (2\pi n t/L).
		\end{align*}
		Then, the equation $-J_{\gamma}(V)=-\lambda V$ gives, for $i=1,2,3$,
		\begin{align*}
			\frac{4\pi^2 n^2}{L^2} V_{i,n} - \sum_{j=1}^3\frac{1}{2 \phi^2} \frac{\partial^2 \phi}{\partial \mu_i \partial \mu_j} V_{j,n}  = -\lambda V_{i,n} ,
		\end{align*}
		for all $n \in \mathbb{N}$ for which $V_{i,n}$ is nonzero, and  
		\begin{align*}
			- \frac{1}{2 \phi^2} \frac{\partial^2 \phi}{\partial \mu_i \partial \mu_j} V_{j,0}  = -\lambda V_{i,0} .
		\end{align*}
		At this point it is convenient to write $\phi = m + \psi$ where $\psi=\sum_{i=1}^k \frac{1}{2|x-P_i|}$ is independent of $m$. Then, the previous equations can be rewritten as
		\begin{align*}
			\frac{1}{2 \phi^2} \frac{\partial^2 \phi}{\partial \mu_i \partial \mu_j} V_{j,0} & = \lambda V_{i,0} \\
			\frac{1}{2 \phi^2} \frac{\partial^2 \phi}{\partial \mu_i \partial \mu_j} V_{j,n} & = \left( \lambda + \frac{4\pi^2 n^2}{L^2} \right) V_{i,n} ,
		\end{align*}
		or, explicitly in terms of $m$ and $\psi$,
		\begin{align*}
			\frac{\partial^2 \psi}{\partial \mu_i \partial \mu_j} V_{j,0} & = 2 (m+\psi)^2 \lambda V_{i,0} \\
			\frac{\partial^2 \psi}{\partial \mu_i \partial \mu_j} V_{j,n} & = 2 (m+\psi)^2 \left( \lambda + n^2(m+\psi) \right) V_{i,n} ,
		\end{align*}
		where we have used the fact that $L^2=4\pi^2 (m+\psi)^{-1}$. Now, notice that the left hand side of these equations is independent of $m$ while for $n \neq 0$ the right hand side increases with $m$ due to the non-negativity of $\lambda$ and the positivity of $n^2m$. Hence, there is $m_0\geq 0$ such that for all $m \geq m_0$ there is no nonzero solution to the equations with $n \neq 0$. We must therefore focus on the equation for $n=0$ and write $\lambda :=  \frac{\mu}{2(m+\psi)^2} $, then
		\begin{align*}
			\frac{\partial^2 \psi}{\partial \mu_i \partial \mu_j} V_{j,0} & = \mu V_{i,0}  
		\end{align*}
		and we conclude that $\mu$ is an eigenvalue of the Hessian of $\psi$ (which coincides with the Hessian of $\phi$) and $(V_{1,0},V_{2,0},V_{3,0})$ an associated eigenvector. Furthermore, given that $\lambda\geq 0$ we must also have that $\mu\geq 0$.
	\end{proof}
	
	\begin{remark}\label{rem:rescaled statement}
		By rescaling the metric using $m^{-1}$ we can equally rescale the (scale invariant) Euclidean metric in $\mathbb{R}^3$. This allows one to state proposition \ref{prop:Jacobi eigenvalues} with $m=1$. Then, the statement is that there is $d_0 >0$ such that if $|P_i-P_j| \geq d_0$ for all $i\neq j$, the same conclusion holds.
	\end{remark}

	\subsubsection{With base $\mathbb{T}_\Lambda$}
	
	In this subsection we perform a similar computation, but with base a $3$-torus $\mathbb{T}_\Lambda$. One key difference is that we now have negative, as well as positive, charges respectively placed at the points $\lbrace q_1, \ldots q_8 \rbrace$ and $\lbrace p_1 , \ldots , p_n \rbrace$ as in \ref{sec:Foscolo K3}. As in that Section, we  have the respective harmonic function $h$ which we use to construct the hyperK\"ahler metric $g_\epsilon^{gh}$ as in \ref{eq:g_epsilon^{gh}}, which corresponds to $\epsilon$ times the Gibbons--Hawking metric determined by the potential function $\phi= \epsilon^{-1}+h$.
	
	As before, we shall now relate the index of a geodesic of the form $\gamma_0=\pi^{-1}(p)$ for the metric $g_\epsilon^{gh}$ with the index of $h$ at $p$.
	
	\begin{proposition}\label{prop:Jacobi eigenvalues 2}
		There is $\epsilon_0>0$ such that for all $\epsilon \leq \epsilon_0$, the following holds. Let $p \in \mathbb{T}^3$ be a critical point of $h$ and $V \in \mathbb{R}^3$ an eigenvector of $\mathrm{Hess}_p(h)$ associated with a non-negative eigenvalue $\mu\geq 0$. Then, for $\gamma_0=\pi^{-1}(p)$, $V$ gives rise to a normal eigenvector-field of $-J_{\gamma_0}$ associated with the eigenvalue
		$$- \frac{\epsilon}{2(1+\epsilon h(p))^2} \mu .$$
		Conversely, any non-positive eigenvector of $-J_{\gamma_0}$ arises in this way. In particular, we have
		$$\mathrm{index}\left(-J_{\gamma_0}\right) = 3- \mathrm{index} \Hess_p(h).$$
	\end{proposition}
	\begin{proof}
		The proof follows from performing the same computations as in the proof of Proposition \ref{prop:Jacobi eigenvalues}, but using instead the metric $g_\epsilon^{gh}$ which yields instead the equations
		\begin{align*}
			\frac{1}{2(1+\epsilon h)^2} \frac{\partial^2 h}{\partial \mu_i \partial \mu_j} V_{j,0} & = \frac{\lambda}{\epsilon} V_{i,0} \\
			\frac{1}{2(1+\epsilon h)^2} \frac{\partial^2 h}{\partial \mu_i \partial \mu_j} V_{j,n} & = \left( \frac{\lambda}{\epsilon} + \frac{4\pi^2 n^2}{\epsilon^3} (1+\epsilon h) \right) V_{i,n} .
		\end{align*}
		Notice that for $n \neq 0$ the rights hand side is exploding due to the $n^2 \epsilon^{-3}$ term, while the left hand side can be bounded above independently of $\epsilon<1$. Hence, for small enough $\epsilon$ we can focus on the first equation, i.e. the one with $n=0$, and write $\lambda := \lambda(\epsilon) = \frac{\epsilon}{2(1+\epsilon h)^2} \mu $
		\begin{align*}
			\frac{\partial^2 h}{\partial \mu_i \partial \mu_j} V_{j,0} & = \mu V_{i,0}  .
		\end{align*}
		Hence, $\mu$ is an eigenvalue of the Hessian of $h$ and $(V_{1,0},V_{2,0},V_{3,0})$ an associated eigenvector. Given that $\lambda \geq 0$ we must also have that $\mu \geq 0$ and we have proved. The converse reasoning applies immediately and we have thus concluded the proof of the stated result.
	\end{proof}

	\subsection{Inverting the Jacobi operator of a geodesic orbit for $g_\epsilon^{gh}$}\label{ss:Jacobi in the building blocks}
	
	In this subsection we shall prove Proposition \ref{prop:Jacobi} which gives conditions upon which it is possible to invert the Jacobi operator of a geodesic orbit for the Gibbons--Hawking metric $g_\epsilon^{gh}$. The result also gives an upper bound on the norm of the inverse which will be necessary later when deforming such a geodesic orbit for $g_\epsilon^{gh}$ to a geodesic of $g_\epsilon^{hk}$. 
	
	\begin{proposition}\label{prop:Jacobi}
		Let $\gamma_0=\pi^{-1}(p)$ be a geodesic circle orbit for the Gibbons--Hawking metric $g^{gh}_\epsilon =\epsilon^2h_\epsilon^{-1} \theta^2 + h_\epsilon g_E$ on $P$ from section \ref{sec:Foscolo K3}. If the Hessian of $h$ is non-degenerate at $p$, then the Jacobi operator of $\gamma_0$
		$$J_{\gamma_0}:W^{2,2}(S^1, \gamma_0^* (TP)^\perp) \to L^{2}(S^1, \gamma_0^* (TP)^\perp),
		$$ 
		is invertible and its inverse satisfies
		$$\| J_{\gamma_0}^{-1}(V) \|_{L^2} \lesssim \epsilon^{-1} \| V \|_{W^{2,2}}. $$
		In particular, $\gamma_0$ is nongenerate, i.e. has nullity zero.
	\end{proposition}
	
	\begin{proof}
		As before, we write a normal vector field $V=\sum_{i=1}^3 V_i E_i$ and write each coefficient $V_i$ as a cosine series 
		$$V_i=  V_{i,0} + \sum_{n \in \mathbb{N}} V_{i,n} \cos \left( 2\pi n \frac{\sqrt{1+\epsilon h}}{\epsilon} t \right),$$ 
		and using equation \ref{eq:Jacobi operator 2}, we find that $J_{\gamma_0}(V)= \sum_{j=1}^3 (J_{\gamma_0}(V))^j E_j$, with
		\begin{align*}
		(J_{\gamma_0}(V))^j  & =   \sum_{i=1}^3 \frac{\epsilon V_{i,0}
			}{2(1+\epsilon h)^2} \frac{\partial^2 h}{\partial \mu_i \partial \mu_j} + \\
		& \ \ \  + \sum_{i=1}^3 \sum_{n \in \mathbb{N}}  V_{i,n} \cos \left( 2\pi n \frac{\sqrt{1+\epsilon h}}{\epsilon} t \right) \left( \frac{\epsilon}{2(1+\epsilon h)^2} \frac{\partial^2 h}{\partial \mu_i \partial \mu_j} - \delta_{ij} \frac{4\pi^2 n^2}{\epsilon^2} (1+\epsilon h) \right) .
		\end{align*}
		At this point it is convenient to split the normal bundle along $\gamma_0$ as $\gamma_0^* (TP)^\perp = T_0 \oplus T_p$ where $T_0= \lbrace V \ | \ \nabla_{\dot{\gamma_0}} V=0 \rbrace$ and $T_p$ denotes its orthogonal complement with respect to the $L^2$-metric. With these splittings in place, notice that 
		$$V_{i,0}E_i \in T_0 , \ \text{and} \ V_{i,n} \cos \left( 2\pi n \frac{\sqrt{1+\epsilon h}}{\epsilon} t \right) E_i \in T_p \ \text{for all $n \in \mathbb{N}$.}$$ 
		Hence, we find that $J_{\gamma_0}|_{T_0}$ is invertible provided that the Hessian of $h$ is. In this case, we have $\|J_{\gamma_0}^{-1} |_{T_0} \| \sim \epsilon^{-1}$. Furthermore, we have that for $\epsilon \ll 1$ we actually have
		$$-\langle V , J_{\gamma_0}(V) \rangle_{L^2} \gtrsim \epsilon^{-2} \|V \|_{L^2}^2 , $$
		for all $V \in T_p$ and so $J_{\gamma_0}|_{T_p}:W^{2,2} \to L^2$ is also invertible and we shall now compute an upper bound for $\|J_{\gamma_0}^{-1}|_{T_p}\|$. Next, we decompose $T_p = \oplus_{n \in \mathbb{N}} T_n$ where $T_n$ is generated by vector fields of the form
		$$V_n =  \cos \left( 2\pi n \frac{\sqrt{1+\epsilon h}}{\epsilon} t \right) ( a_1 E_ 1 + a_2 E_2 + a_3 E_3 ), $$ 
		for $a_1,a_2,a_3 \in \mathbb{R}$. In particular, given that we are at a critical point of $h$ we have
		\begin{align*}
			|V_n|^2 & = (a_1^2+a_2^2+a_3^2) \cos \left( 2\pi n \frac{\sqrt{1+\epsilon h}}{\epsilon} t \right)^2 \\
			|\nabla_{E_0} V_n|^2 & = 4 \pi^2 n^2  \frac{1+\epsilon h}{\epsilon^2} (a_1^2 + a_2^2 +a_3^2)  \cos \left( 2\pi n \frac{\sqrt{1+\epsilon h}}{\epsilon} t \right)^2 \\
			|\nabla_{E_0}^2 V_n|^2 & = 16 \pi^4 n^4  \frac{(1+\epsilon h)^2}{\epsilon^4} (a_1^2 + a_2^2 +a_3^2)  \cos \left( 2\pi n \frac{\sqrt{1+\epsilon h}}{\epsilon} t \right)^2 ,
		\end{align*}
		From this we find that for $\epsilon \ll 1$
		$$\|V_n \|_{W^{2,2}} \sim n \epsilon^{-1} \| V_n \|_{W^{2,1}} \sim n^2 \epsilon^{-2} \|V_n \|_{L^2},$$
		and we find from the previous computation that 
		$$\| J_{\gamma_0}(V_n) \|_{L^2} \gtrsim n^2\epsilon^{-2} \| V_n \|_{L^2} \sim n^2\epsilon^{-2} \left( n^{-2} \epsilon^2 \| V_n \|_{W^{2,2}} \right) \sim \| V_n \|_{W^{2,2}}.$$
		Again, from this we discover that $J_{\gamma_0}|_{T_p}:W^{2,2} \to L^2$ is invertible and
		$$\| J_{\gamma_0}|_{T_p}^{-1}(W) \|_{W^{2,2}} \lesssim \| W \|_{L^2},$$
		for all $W \in T_p$. Hence, we can construct an inverse $J_{\gamma_0}^{-1} = J_{\gamma_0}^{-1}|_{T_0} \oplus J_{\gamma_0}^{-1}|_{T_p}:L^2(T) \to W^{2,2}(T)$ satisfying 
		$$\| J_{\gamma_0}^{-1}(V_0 + V_p) \|_{L^2} \lesssim \epsilon^{-1} \| V_0 \|_{W^{2,2}} + \| V_p \|_{W^{2,2}} \lesssim \epsilon^{-1} \| V \|_{W^{2,2}}, $$
		for $\epsilon \ll 1$.
	\end{proof}

	\section{Geodesics on $\bigcup_{j=1}^8 M_\epsilon^j \cup \bigcup_{i=1}^n N_\epsilon^i$}\label{sec:deform geodesics concentrated}
	
	In this section we shall explain how to use what we have learned about geodesics on ALF gravitational instantons in order to construct closed geodesics on the curvature concentrated region of Foscolo's K3 surfaces, namely $\bigcup_{j=1}^8 M_\epsilon^j \cup \bigcup_{i=1}^n N_\epsilon^i$. In both cases, in $M_\epsilon^j$ and $N_\epsilon^i$, we shall deform the original geodesics on the ALF gravitational instantons by making use of White's perturbation theorem.

	\subsection{Geodesics on $M_\epsilon^j$}\label{ss:geodesics on AH}
	
	For each $j=1, \ldots , 8$, consider the restriction of the hyperK\"ahler metric $\epsilon^{-2}g_\epsilon^{hk}$ to $M_\epsilon^j$. We mentioned in section \ref{sec:Foscolo K3} that, with this scaling, the metric $\epsilon^{-2}g_\epsilon^{hk}$ uniformly converges with all derivatives to an initially fixed ALF gravitational instanton of type $D_{m_j}$ denoted by $M^j$.
	
	Then, it follows from White's deformation theorem (Theorem 3.2 in \cite{White}) that any non-degenerate closed geodesic in $M^j$ deforms to a non-degenerate closed geodesic in $M_\epsilon^j$ with the same index. Unfortunately, the author is not knowledgeable about closed geodesics on ALF gravitational instanton of type $D_{m}$ for $m \geq 1$. On the other hand, in the case $m=0$, i.e. the Atiyah--Hitchin manifold, we saw in section \ref{sec:D0} that there is no closed geodesic which is stable to second order, but unstable closed geodesics do exist. Indeed, we saw that there is a family of closed geodesics parameterized by $\mathbb{RP}^2$. Then, as $\mathbb{RP}^2$ has Lusternik--Schnirelmann number $3$, by White's deformation theorem \cite{White} there are at least $3$ closed geodesics on $M_\epsilon^j$ for sufficiently small $\epsilon>0$.

	\subsection{Geodesics on $N_\epsilon^i$}
	
	Recall, from section \ref{sec:Foscolo K3} that for each $i=1, \ldots , 8$, the rescaled metrics $\epsilon^{-2}g_\epsilon^{hk}$ restricted to $N_\epsilon^i$ uniformly converge with all derivatives to an initially fixed ALF gravitational instanton of type $A_{k_i-1}$ denoted $N^i$.
	
	As before, for small enough $\epsilon$, we can use White's deformation theorem to deform any non-degenerate closed geodesic in $N^i$ to one in $N_\epsilon^i$ of the same index. Hence, it is sufficient to find such geodesics for $N^i$.
	
	\begin{example}\label{ex:A_2}
		Let $N^i$ be the ALF gravitational instanton of type $A_2$ constructed using the Gibbons--Hawking ansatz by placing the singularities $\lbrace P_1, P_2, P_3 \rbrace$ of $\phi(x)=m+\sum_{i=1}^3 \frac{1}{2 |x-P_i|}$ in the vertices of an equilateral triangle. By proposition \ref{prop:triangle} there are $4$ non-degenerate geodesics. Together with proposition \ref{prop:Jacobi eigenvalues} we conclude that for either $m \geq m_0$ or $\min_{i \neq j}|P_i-P_j| >d_0$, one of these geodesics has index $1$ while the remaining three have index $2$.
	\end{example}

	\begin{example}\label{ex:A_8}
		Consider several groups of three singularities $\lbrace P^{(n)}_1 , P^{(n)}_2, P^{(n)}_3 \rbrace_{n=1, \ldots , N} $ with each group placed at a sufficiently large distance from any other. We shall explicitly impose this by demanding that 
		$$D:= \min_{n,m=1, \ldots , N} \min_{i,j=1,2,3} |P_i^{(n)}-P_j^{(m)}| \gg d:= \max_{n=1, \ldots , N} \max_{i,j=1,2,3} |P_i^{(n)}-P_j^{(n)}|.$$
		Then, we can write $\phi = m + \sum_{n=1}^N \phi^{(n)}$ with
		$$\phi^{(n)} = \sum_{i=1}^3 \frac{1}{2|x-P^{(n)}_i|}. $$
		Then, for any $l \in \lbrace 1, \ldots , N \rbrace$ and $x$ in the convex hull of the $3$ points $\lbrace P_1^{(l)},P_2^{(l)},P_3^{(l)} \rbrace$, we find that for $n \neq l$
		$$|\nabla^k \phi^{(n)}| \lesssim D^{-k-1}.$$
		In particular, if $\phi^{(l)}$ is non-degenerate, we find that there is $D_0$ sufficiently large so that for $D \geq D_0$ the function $\phi$ has the same number of critical points in the convex hull of $\lbrace P_1^{(l)},P_2^{(l)},P_3^{(l)} \rbrace$ as $\phi^{(l)}$ and these have the same index.
		
		With these observations in mind we now set $N=3$ and place each $\lbrace P^{(n)}_1 , P^{(n)}_2, P^{(n)}_3 \rbrace$ at the vertices of equilateral triangles of side $d$. Then, we obtain three index $2$ and one index $1$ critical points located in the interior of each of these triangles. Furthermore, we can place each of these small triangles at the vertices of a much larger triangle (of side $D \gg d$). Then, in the interior of this large triangle we find three more index $2$ critical points and one more of index $1$. 
		
		Overall, using proposition \ref{prop:Jacobi eigenvalues} we find that in the resulting ALF gravitational instanton of type $A_8$, named $N^i$, there is a total of $12$ geodesic orbits of index one and $4$ geodesic orbits of index two. See Figure \ref{fig:tritriangle}.
	\end{example}
	
	\begin{example}\label{ex:A_1}
		Let $N^i$ be an ALF gravitational instanton of type $A_1$. Then, the pre-image via $\pi$ of the straight line connecting the two singularities of $\phi$ is a totally geodesic, non-round, but axially symmetric ellipsoid. Therefore, it admits $1$ nondegenerate closed geodesics and a $1$-parameter family of closed geodesics, parametrized by $S^1$, with nullity equal to one and index two. 
		As the Lusternik--Schnirelmann number of $S^1$ is equal to two, we find by White's deformation theorem \cite{White} that any sufficiently small perturbation of this metric admits at least $1+2=3$ closed geodesics, each of each index one, two and three.
	\end{example}

	\section{Geodesics on $K_\epsilon$}\label{sec:deform geodesics collapse}
	
	This section concerns the construction of closed geodesics on the collapsing region of Fosocolo's K3 surfaces, i.e. in $K_\epsilon$. We start in \ref{ss:Deforming in K epsilon} by deforming collapsing geodesic circles in the Gibbons--Hawking approximation $K_{\epsilon^{gh}}$ to actual geodesics on Fosocolo's K3-surface $K_\epsilon$. The main result of this subsection is Proposition \ref{prop:solving the geodesic equation}. Then, in \ref{ss:Critical points of h} we investigate critical points of $h$ as these give rise to closed geodesics on the approximation $K_{\epsilon^{gh}}$. When non-degenerate, these can be fed to Proposition \ref{prop:solving the geodesic equation} and thus deformed to closed geodesics of $K_\epsilon$. Finally, in \ref{ss:cube} we give an example that serves as an application of the remaining results of this section.
	
	Before dwelling into the statement and proof of Proposition \ref{prop:solving the geodesic equation} we recall some notation which is useful in stating that result. As mentioned in section \ref{sec:Foscolo K3}, $K_\epsilon$ is a large open set in $X_\epsilon=(X, g_\epsilon^{hk})$ which we can also view as a subset of $P/\mathbb{Z}_2$ equipped with the hyperK\"ahler metric $g_\epsilon^{hk}$. Hence, we can restrict the projection $\pi$ to $K_\epsilon$ yielding a map $\pi: K_\epsilon \to \mathbb{T}_\Lambda/\mathbb{Z}_2$.

	\subsection{Deforming geodesics}\label{ss:Deforming in K epsilon}
	
	In this subsection we shall deform geodesics with respect to the approximate hyperK\"ahler metric $g_\epsilon^{gh}$ to actual geodesics for Foscolo's K3 metric $g_\epsilon^{hk}$.
	
	\begin{proposition}\label{prop:solving the geodesic equation}
		Let $p$ be a nondegenerate critical point of $h$ and $\gamma_0=\pi^{-1}(p)$. Then, there is $\epsilon_0>0$ with the property that for all positive $\epsilon < \epsilon_0$:
		\begin{itemize}
			\item There is a unique normal vector field $U \in \Gamma(S^1, \gamma_0^*(TX)^\perp)$ satisfying $\|U\|_{W^{2,2}} \lesssim \epsilon^{9/5}$ and such that $\gamma(t)=\exp_{\gamma_0(t)}(U(t))$ is a closed geodesic for the hyperK\"ahler metric $g_{\epsilon}^{hk}$;
			
			\item The Hausdorff distance between $\gamma_0$ and $\gamma$ satisfies $d_H (\gamma_0 , \gamma)  \lesssim \epsilon^{13/10} $.
			
			\item The index of the $\gamma$ with respect to $g_\epsilon^{hk}$ coincides with that of $\gamma_0$ with respect to $g_\epsilon$.
		\end{itemize} 
	\end{proposition}
	
	Before diving into the proof, it is convenient to first compute how far $\gamma_0$ is from being a geodesic with respect to $g_\epsilon^{hk}$ and compute some estimates on the linearized geodesic equation. 
	
	Let $s_0>0$ be a small parameter to be fixed at a later stage and $s \in (0,s_0)$. Then, we write the geodesic equation as the vanishing locus of the function
	$$f_s:W^{2,2}(S^1 , \gamma_0^* TX^\perp) \to L^{2}(S^1 , \gamma_s^* TX^\perp)$$
	given by
	$$f_s(V)=(\nabla^{g_\epsilon^{hk}}_{\dot{\gamma}_s} \dot{\gamma_s})^{\perp_{g_\epsilon^{hk}} },$$ 
	where $\gamma_s(t)=\exp_{\gamma_0(t)}(sV(t))$, $(\cdot)^{\perp_{g_\epsilon^{hk}} }$ denotes the projection on the normal bundle to $\langle \dot{\gamma}_s \rangle \subset TX|_{\gamma_s(S^1)}$ using the metric $g_\epsilon^{hk}$, and $\nabla^{g_\epsilon^{hk}}$ is its Levi-Civita connection.

	\subsubsection*{The error term}

	Recall from equation \ref{eq:metric estimate} that 
	\begin{equation}\label{eq:compare metrics}
		g_\epsilon^{hk}= g_\epsilon + \epsilon^{\alpha} \Upsilon_\epsilon ,
	\end{equation}
	for $\Upsilon_\epsilon$ a symmetric tensor satisfying $|\Upsilon_\epsilon|_{g_\epsilon} \lesssim 1$ with derivatives and $\alpha=\frac{11-2\delta}{5}$ with $\delta \in (-1/2,0)$. Now, choose a local coordinate system $\lbrace x_\mu \rbrace_{\mu=1}^4$ such that $|\partial_{x_\mu}|_{g_\epsilon^{hk}} \lesssim 1$ (for sufficiently small $\epsilon \ll 1$ this also implies $|\partial_{x_\mu}|_{g_\epsilon} \lesssim 1$) for $\mu=1,2,3,4$. In the domain of this coordinate system reads, the geodesic equation for $t \mapsto \gamma_0(t)$ with respect to $g_\epsilon^{hk}$ reads
	\begin{align*}
		 \nabla_{\dot{\gamma}_0}^{g_\epsilon^{hk}} \dot{\gamma}_0  & = \left( \ddot{\gamma}_0^\nu + ((\Gamma^{hk}_\epsilon)^\nu_{\mu \lambda} \circ \gamma_0 ) \dot{\gamma_0}^\mu \dot{\gamma_0}^\lambda \right) \partial_{x_\nu},
	\end{align*}
	where $(\Gamma^{hk}_\epsilon)^\nu_{\mu \lambda}$ denotes the Christoffel symbols of the metric $g_\epsilon^{hk}$. Recall that, for a generic metric $g$, these are computed using the metric via the formula
	$$\Gamma^\nu_{\mu \lambda} = \frac{1}{2} g^{\nu \rho} \left( \partial_{\lambda} g_{\rho \mu} + \partial_\mu g_{\rho \lambda} -\partial_\rho g_{\mu \lambda} \right).$$
	Hence, using equation \ref{eq:compare metrics} we find that
	\begin{equation}\label{eq:Christoffel symbols}
		(\Gamma^{hk}_\epsilon)^\nu_{\mu \lambda} = (\Gamma_\epsilon)^\nu_{\mu \lambda} + \epsilon^\alpha (\delta \Gamma_\epsilon)^\nu_{\mu \lambda}, 
	\end{equation}
	where $ (\Gamma_\epsilon)^\nu_{\mu \lambda}$ denotes the Christoffel symbols of the metric $g_\epsilon$ and $|(\delta \Gamma_\epsilon)^\nu_{\mu \lambda}| \lesssim 1$. Inserting this into the geodesic equation yields
	\begin{align*}
		\nabla_{\dot{\gamma}_0}^{g_\epsilon^{hk}} \dot{\gamma}_0  & = \left( \ddot{\gamma}_0^\nu + ((\Gamma_\epsilon)^\nu_{\mu \lambda} \circ \gamma_0 ) \dot{\gamma_0}^\mu \dot{\gamma_0}^\lambda \right) \partial_{x_\nu} + \epsilon^\alpha \left(  ((\delta \Gamma_\epsilon)^\nu_{\mu \lambda} \circ \gamma_0 ) \dot{\gamma_0}^\mu \dot{\gamma_0}^\lambda \right) \partial_{x_\nu} \\
		& = \epsilon^\alpha \left(  ((\delta \Gamma_\epsilon)^\nu_{\mu \lambda} \circ \gamma_0 ) \dot{\gamma_0}^\mu \dot{\gamma_0}^\lambda \right) \partial_{x_\nu} ,
	\end{align*}
	as $\gamma_0$ is a geodesic for the metric $g_\epsilon$. Given that $|\partial_{x_\mu}| \lesssim 1$ have therefore found that
	\begin{equation}\label{eq:C0 estimate error term} 
		|\nabla_{\dot{\gamma}_0}^{g_\epsilon^{hk}} \dot{\gamma}_0 |_{g_\epsilon^{hk}} \lesssim \epsilon^\alpha ,
	\end{equation}
	and similarly we find that $|\nabla_{\dot{\gamma}_0}^{g_\epsilon^{hk}} \dot{\gamma}_0 |_{g_\epsilon} \lesssim \epsilon^\alpha$.

	Furthermore, the $L^2$-metric on $S^1$ induced by $g_\epsilon$ satisfies $\gamma_0^* g_\epsilon = \frac{\epsilon^2}{1+\epsilon h} \gamma_0^*\theta^2=dt^2$ for $t \in [0, 2\pi \frac{\epsilon}{\sqrt{1+\epsilon h}} ]$. However, it is at times convenient to work with a parameter which is independent of $\epsilon$. Hence, we rescale the coordinate $t$, which is the arclength parameter, to $u = \frac{\sqrt{1+\epsilon h}}{\epsilon} t \in [0,2\pi]$ using which the metric becomes $\gamma_0^*g_{\epsilon}=\frac{\epsilon^2}{1+\epsilon h} du^2$. Furthermore, using again $g_\epsilon^{hk}= g_\epsilon + \epsilon^{\alpha} \Upsilon_\epsilon$, we find that $\gamma_0^* g_\epsilon^{hk} = \frac{\epsilon^2}{1+\epsilon h} du^2 + O(\epsilon^\alpha)$. Thus, as $\alpha >2$, for sufficiently small $\epsilon$ we can estimate the $L^2$-norm of $\nabla_{\dot{\gamma}_0}^{g_\epsilon^{hk}} \dot{\gamma}_0$ as follows 
	\begin{align}\label{eq:L2 estimate error term} 
		\|\nabla_{\dot{\gamma}_0}^{g_\epsilon^{hk}} \dot{\gamma}_0 \|_{L^2} & \lesssim \left( \int_0^{2\pi}  |\nabla_{\dot{\gamma}_0}^{g_\epsilon^{hk}} \dot{\gamma}_0 |_{g_\epsilon}^2 \epsilon du  \right)^{\frac{1}{2}} \lesssim \epsilon^{\alpha + \frac{1}{2}}.
	\end{align}

	\subsubsection*{Comparing Jacobi operators}\label{ss:Comparing Jacobi operator}
	
	We start by recalling the definition of the Jacobi operator $J^{g}_{\gamma_0}(\cdot)$ of a curve $\gamma_0$ with respect to a generic metric $g$. By definition, this is given by
	$$J_{\gamma_0}(V):=\left( \frac{D}{dt} \frac{DV}{dt} - R^g \left(\dot{\gamma}_0 , V \right) \dot{\gamma}_0 \right)^\perp,$$
	with $R^g(\cdot, \cdot) \cdot$ the Riemann curvature tensor of $g$ and $(\cdot)^\perp$ the orthogonal component with respect to the metric $g$. Our next goal is to compare $J_{\gamma_0}^{g_\epsilon}$ and $J_{\gamma_0}^{g_\epsilon^{hk}}$ and to further compare these with the linearization of the geodesic equation for $g^{hk}_\epsilon$ at $\gamma_0$.
	
	Using Koszul formula, we find that for any vector fields $X,Y,Z$
	\begin{align*}
		2g^{hk}_\epsilon (\nabla^{g_\epsilon^{hk}}_X Y, Z) & = X \left( g_\epsilon^{hk} (Y,Z) \right) + Y \left( g_\epsilon^{hk} (Z,X) \right) - Z \left( g_\epsilon^{hk} (X,Y) \right) \\
		& \ \ \ \ + g_\epsilon ([X,Y],Z) + g_\epsilon ([Z,X],Y) - g_\epsilon ([Y,Z],X) \\
		& = X \left( g_\epsilon (Y,Z) \right) + Y \left( g_\epsilon (Z,X) \right) - Z \left( g_\epsilon (X,Y) \right) \\
		& \ \ \ \ + g_\epsilon ([X,Y],Z) + g_\epsilon ([Z,X],Y) - g_\epsilon ([Y,Z],X) \\
		& \ \ \ \ + \epsilon^\alpha \left( X \left( \Upsilon_\epsilon (Y,Z) \right) + Y \left( \Upsilon_\epsilon (Z,X) \right) - Z \left( \Upsilon_\epsilon (X,Y) \right) \right) \\
		& \ \ \ \ + \epsilon^\alpha \left( \Upsilon_\epsilon ([X,Y],Z) + \Upsilon_\epsilon ([Z,X],Y) - \Upsilon_\epsilon ([Y,Z],X) \right) \\
		& = 2g_\epsilon (\nabla^{g_\epsilon}_X Y, Z) + \epsilon^\alpha \left( \tilde{\Upsilon}_\epsilon (X,Z) \right) (Y),
	\end{align*} 
	where
	\begin{align*}
		\left( \tilde{\Upsilon}_\epsilon (X,Z) \right) (Y) & =   X \left( \Upsilon_\epsilon (Y,Z) \right) + Y \left( \Upsilon_\epsilon (Z,X) \right) - Z \left( \Upsilon_\epsilon (X,Y) \right)  \\
		& \ \ \ \ +  \Upsilon_\epsilon ([X,Y],Z) + \Upsilon_\epsilon ([Z,X],Y) - \Upsilon_\epsilon ([Y,Z],X) ,
	\end{align*}
	is tensorial in $X,Z$ but not in $Y$. Furthermore, if $\lbrace e_\mu \rbrace_{\mu=1}^4$ denotes an orthonormal framing for the metric $g_\epsilon^{hk}$, we can write 
	\begin{align*}
		\left( \tilde{\Upsilon}_\epsilon (X,Z) \right) (Y) & = \left( \tilde{\Upsilon}_\epsilon (X,e_\mu) \right) (Y) \left( g_\epsilon^{hk} ( e_\mu, Z) \right) \\
		& = g_\epsilon^{hk} \left(  \left( \tilde{\Upsilon}_\epsilon (X,e_\mu) \right) (Y) e_\mu, Z \right),
	\end{align*}
	and
	\begin{align*}
		g_\epsilon (\nabla^{g_\epsilon}_X Y, Z) & = g_\epsilon^{hk} (\nabla^{g_\epsilon}_X Y, Z) - \epsilon^\alpha \Upsilon_\epsilon( \nabla^{g_\epsilon}_X Y , Z) \\
		& = g_\epsilon^{hk} (\nabla^{g_\epsilon}_X Y, Z) - \epsilon^\alpha g^{hk}_\epsilon \left( \Upsilon_\epsilon( \nabla^{g_\epsilon}_X Y , e_\mu) e_\mu, Z \right) \\
		& = g_\epsilon^{hk} \left( \nabla^{g_\epsilon}_X Y  - \epsilon^\alpha  \Upsilon_\epsilon( \nabla^{g_\epsilon}_X Y , e_\mu) e_\mu, Z \right) .
	\end{align*}
	Finally, inserting both of these into the computation above for $2g^{hk}_\epsilon (\nabla^{g_\epsilon^{hk}}_X Y, Z) $ yields
	\begin{align*}
		2g^{hk}_\epsilon (\nabla^{g_\epsilon^{hk}}_X Y, Z) & = 2g_\epsilon \left(   \nabla^{g_\epsilon}_X Y  - \epsilon^\alpha  \Upsilon_\epsilon( \nabla^{g_\epsilon}_X Y , e_\mu) e_\mu + \frac{\epsilon^\alpha}{2}  \left( \tilde{\Upsilon}_\epsilon (X,e_\mu) \right) (Y) e_\mu , Z  \right) .
	\end{align*} 
	As this is valid for all $Z$ we conclude that
	\begin{equation*}
		\nabla^{g_\epsilon^{hk}}_X Y = \nabla^{g_\epsilon}_X Y  - \epsilon^\alpha  \Upsilon_\epsilon( \nabla^{g_\epsilon}_X Y , e_\mu) e_\mu + \frac{\epsilon^\alpha}{2}  \left( \tilde{\Upsilon}_\epsilon (X,e_\mu) \right) (Y) e_\mu  .
	\end{equation*}
	The exact formula of this equation will not be necessary for us and all we will need is that for any vector fields $X,Y$ we have
	$$\nabla^{g_\epsilon^{hk}}_X Y = \nabla^{g_\epsilon}_X Y + \epsilon^{\alpha} O_X(Y),$$
	for $O_X(Y)$ a quantity which is tensorial in $X$ but not in $Y$ and which is of order $O(1)$ with derivatives, independently of $\epsilon \ll 1$. Inserting this into the definition of the Riemann curvature tensor we find that the Riemann curvature tensors of these two metrics are related via
	$$R^{g_\epsilon^{hk}}(\cdot , \cdot ) \cdot = R^{g_\epsilon}(\cdot , \cdot ) \cdot + \epsilon^{\alpha} O_R(\cdot , \cdot , \cdot),$$
	with again $O_R(\cdot , \cdot , \cdot) \sim O(1)$. From these two observations we find that the Jacobi operators of $\gamma_0$ with respect to $g_\epsilon^{hk}$ and $g_\epsilon$ differ by a term of order $\epsilon^\alpha$, i.e. for any vector field $V$
	\begin{equation}\label{eq:comparing Jacobi operator}
		J_{\gamma_0}^{g_\epsilon^{hk}}(V) = J_{\gamma_0}^{g_\epsilon}(V) + \epsilon^\alpha \delta J(V), 
	\end{equation}
	where $|\delta J(V)| \lesssim |V|$ pointwise.

	\subsubsection*{Linearized geodesic equation}
	
	However, the linearization of the geodesic equation $f_s(V)=0$ is not exactly the Jacobi operator because $t \mapsto \gamma_0(t)$ is not a geodesic for the metric $g^{hk}_\epsilon$. The next lemma computes the linearization of the geodesic equation.
	
	\begin{lemma}\label{lem:perpendicular}
		Let $\gamma_s(t)=\exp_{\gamma_0(t)}(sV(t))$ as before. Then, 
		\begin{align*}
			\frac{D}{ds} \Big\vert_{s=0} (\nabla^{g^{hk}_\epsilon}_{\dot{\gamma}_s} \dot{\gamma}_s )^{\perp_{g^{hk}_\epsilon}} & = J_{\gamma_0}^{g^{hk}_\epsilon}(V)  - \langle \nabla^{g^{hk}_\epsilon}_{\dot{\gamma}_0} \dot{\gamma}_0 , \nabla^{g^{hk}_\epsilon}_{\dot{\gamma}_0}V \rangle_{g^{hk}_\epsilon} \frac{\dot{\gamma}_0}{|\dot{\gamma}_0|_{g^{hk}_\epsilon}^2} - \langle \nabla^{g^{hk}_\epsilon}_{\dot{\gamma}_0} \dot{\gamma}_0 , \dot{\gamma}_0 \rangle_{g^{hk}_\epsilon} \frac{\nabla^{g^{hk}_\epsilon}_{\dot{\gamma}_0}V}{|\dot{\gamma}_0|_{g^{hk}_\epsilon}^2}.
		\end{align*}
	\end{lemma}
	\begin{proof}
		In order to ease notation, during the proof we shall drop $g^{hk}_\epsilon$ from all the sub and super-scripts. Explicitly, we have $(\nabla_{\dot{\gamma}_s} \dot{\gamma}_s )^\perp =\nabla_{\dot{\gamma}_s} \dot{\gamma}_s - \langle \nabla_{\dot{\gamma}_s} \dot{\gamma}_s , \frac{\dot{\gamma}_s}{|\dot{\gamma}_s|} \rangle \frac{\dot{\gamma}_s}{|\dot{\gamma}_s|} $. Then, taking derivatives and evaluating at $s=0$ we find that 
		\begin{align}\nonumber
			\frac{D}{ds}  (\nabla_{\dot{\gamma}_s} \dot{\gamma}_s )^\perp & = \frac{D}{ds} \nabla_{\dot{\gamma}_s} \dot{\gamma}_s  - \langle \frac{D}{ds} \nabla_{\dot{\gamma}_s} \dot{\gamma}_s , \frac{\dot{\gamma}_s}{|\dot{\gamma}_s|} \rangle \frac{\dot{\gamma}_s}{|\dot{\gamma}_s|} - \langle \nabla_{\dot{\gamma}_s} \dot{\gamma}_s , \frac{D}{ds}  \frac{\dot{\gamma}_s}{|\dot{\gamma}_s|} \rangle \frac{\dot{\gamma}_s}{|\dot{\gamma}_s|} - \langle \nabla_{\dot{\gamma}_s} \dot{\gamma}_s , \frac{\dot{\gamma}_s}{|\dot{\gamma}_s|} \rangle \frac{D}{ds} \frac{\dot{\gamma}_s}{|\dot{\gamma}_s|} \\ \label{eq:DDd proof of lemma}
			& = \left( \frac{D}{ds} \nabla_{\dot{\gamma}_s} \dot{\gamma}_s \right)^\perp - \langle \nabla_{\dot{\gamma}_s} \dot{\gamma}_s , \frac{D}{ds}  \frac{\dot{\gamma}_s}{|\dot{\gamma}_s|} \rangle \frac{\dot{\gamma}_s}{|\dot{\gamma}_s|} - \langle \nabla_{\dot{\gamma}_s} \dot{\gamma}_s , \frac{\dot{\gamma}_s}{|\dot{\gamma}_s|} \rangle \frac{D}{ds} \frac{\dot{\gamma}_s}{|\dot{\gamma}_s|}
		\end{align}
		and we shall now compute each term separately. The first is given by
		\begin{align*}
			\left( \frac{D}{ds} \nabla_{\dot{\gamma}_s} \dot{\gamma}_s \right)^\perp & =  \left( \frac{D}{ds} \frac{D}{dt} \frac{d \gamma_s}{dt} \right)^\perp   \\
			& =  \left( \frac{D}{dt} \frac{D}{ds} \frac{d \gamma_s}{dt} - R \left( \frac{d \gamma_s}{dt} ,  \frac{d\gamma_s}{ds}\right) \frac{d \gamma_s}{dt} \right)^\perp  \\ \label{eq:DDd proof of lemma}
			& =  \left( \frac{D}{dt} \frac{DV}{dt} - R \left(\dot{\gamma}_s , V \right) \dot{\gamma}_s \right)^\perp ,
		\end{align*}
		and in order to compute the remaining two terms we need to find $\frac{D}{ds} \frac{\dot{\gamma}_s}{|\dot{\gamma}_s|}$ which is given by
		\begin{align*}
			\frac{D}{ds} \frac{\dot{\gamma}_s}{|\dot{\gamma}_s|} & =  \frac{1}{|\dot{\gamma}_s|} \frac{D}{ds} \dot{\gamma}_s -  \frac{\dot{\gamma}_s}{2|\dot{\gamma}_s|^3} \partial_s |\dot{\gamma}_s|^2 .
		\end{align*}
		On the other hand, we have 
		$$ \partial_s |\dot{\gamma}_s|^2 = 2 \langle \frac{D}{ds}\dot{\gamma}_s , \dot{\gamma}_s \rangle = 2 \langle \frac{DV}{dt} , \dot{\gamma}_s \rangle = 2 \partial_t \langle V, \dot{\gamma}_s \rangle -2 \langle V, \nabla_{\dot{\gamma}_s} \dot{\gamma}_s \rangle= -2 \langle V, \nabla_{\dot{\gamma}_s} \dot{\gamma}_s \rangle$$ 
		because $\langle V, \dot{\gamma}_s \rangle=0$, and inserting this into the above gives
		\begin{align*}
			\frac{D}{ds} \frac{\dot{\gamma}_s}{|\dot{\gamma}_s|} & =  \frac{1}{|\dot{\gamma}_s|} \frac{DV}{dt}  -  \frac{ \langle V, \nabla_{\dot{\gamma}_s} \dot{\gamma}_s \rangle}{|\dot{\gamma}_s|^3} \ \dot{\gamma}_s .
		\end{align*}
		Inserting all this into equation \ref{eq:DDd proof of lemma} above yields
		\begin{align*}
			\frac{D}{ds}  (\nabla_{\dot{\gamma}_s} \dot{\gamma}_s )^\perp & = \left( \frac{D}{dt} \frac{DV}{dt} - R \left(\dot{\gamma}_s , V \right) \dot{\gamma}_s \right)^\perp  - \langle \nabla_{\dot{\gamma}_s} \dot{\gamma}_s , \nabla_{\dot{\gamma}_s}V \rangle \frac{\dot{\gamma}_s}{|\dot{\gamma}_s|^2} - \langle \nabla_{\dot{\gamma}_s} \dot{\gamma}_s , \dot{\gamma}_s \rangle \frac{\nabla_{\dot{\gamma}_s}V}{|\dot{\gamma}_s|^2},
		\end{align*}
		which is the result on the statement.
	\end{proof}

	\begin{proof}[Proof of Proposition \ref{prop:solving the geodesic equation}]

	
	We want to take a Taylor expansion of $f_s(V)$ in $s$. However, in order to achieve this we must regard $f_s$ as taking values in the same space independent of $s$, which can be achieved by trivializing $TX$ in an open neighborhood $U$ of $\lbrace \gamma_0(t)\rbrace_{t}$ containing the image of all possible curves $t \mapsto \gamma_s(t)$ which we shall be considering. In this situation, we have $TX|_U \cong U \times \mathbb{R}^4$ and implicitly composing with this isomorphism we shall regard $f_s$ as a family of maps
	$$f_s:W^{2,2}(S^1 ,  \langle \dot{\gamma}_0 \rangle^\perp ) \to L^{2}(S^1 , \langle \dot{\gamma_s} \rangle^\perp ),$$
	where we view $\dot{\gamma_0}$ and $\dot{\gamma_s}$ as vectors in $\mathbb{R}^4$. Next, using the fact that the curves $\gamma_s$ have trivial normal bundle, we shall fix an identification $\psi_s : \langle \dot{\gamma_s} \rangle^{\perp} \to \langle \dot{\gamma_0} \rangle^{\perp}  \cong \mathbb{R}^3$ and consider the map 
	$$F_s:= \psi_s \circ f_s : W^{2,2}(S^1 ,  \mathbb{R}^3 ) \to L^{2}(S^1 ,  \mathbb{R}^3 ),$$
	with $\psi_0$ the identity. Then, taking a Taylor expansion of $F_s(V)$ as a function of $s$ around $s=0$, gives the geodesic equation
	\begin{align*}
		0 & = F_0(V) + s \frac{d}{ds} \Big\vert_{s=0} F_s(V) + \frac{s^2}{2} R_s(V) \\
		& = (\nabla^{g_\epsilon^{hk}}_{\dot{\gamma}_0} \dot{\gamma_0})^{\perp_{g_\epsilon^{hk}} } + s \left[ \frac{d\psi_s}{ds}\Big\vert_{s=0}\left( (\nabla^{g_\epsilon^{hk}}_{\dot{\gamma}_0} \dot{\gamma_0})^{\perp_{g_\epsilon^{hk}} }  \right) + \frac{D}{ds}\Big\vert_{s=0}(\nabla^{g_\epsilon^{hk}}_{\dot{\gamma}_s} \dot{\gamma_s})^{\perp_{g_\epsilon^{hk}} }  \right] + \frac{s^2}{2} R_s(V) ,
	\end{align*}
	with the remainder term satisfying $\|R_s(V)\|_{L^2} \lesssim \| V \|^2_{W^{1,2}}$, see the Appendix \ref{appendix:geodesic equation remainder}.
	Furthermore, from combining equation \ref{eq:L2 estimate error term} with Lemma \ref{lem:perpendicular}, we find that $|\nabla^{g^{hk}_\epsilon}_{\dot{\gamma}_0} \dot{\gamma}_0| \lesssim \epsilon^{\alpha+1/2}$ and so $\frac{D}{ds}\Big\vert_{s=0}(\nabla^{g_\epsilon^{hk}}_{\dot{\gamma}_s} \dot{\gamma_s})^{\perp_{g_\epsilon^{hk}} }$ differs from $J_{\gamma_0}^{g^{hk}_\epsilon}(V)$ by a quantity of order $\epsilon^{\alpha+1/2}$. On the other hand, equation \ref{eq:comparing Jacobi operator} shows that $J_{\gamma_0}^{g^{hk}_\epsilon}(V)$ differs from $J_{\gamma_0}^{g_\epsilon}(V)$ by a quantity of order $\epsilon^\alpha$. Hence,
	$$\frac{D}{ds}\Big\vert_{s=0}(\nabla^{g_\epsilon^{hk}}_{\dot{\gamma}_s} \dot{\gamma_s})^{\perp_{g_\epsilon^{hk}} }= J_{\gamma_0}^{g_\epsilon}(V) + O(\epsilon^\alpha) .$$
	Furthermore, as $\epsilon^\alpha \ll \epsilon$ it follows from Proposition \ref{prop:Jacobi} that the invertibility of $J_{\gamma_0}^{g_\epsilon}:W^{2,2} \to L^2$ translates into the invertibility of $\frac{D}{ds}\vert_{s=0}(\nabla^{g_\epsilon^{hk}}_{\dot{\gamma}_s} \dot{\gamma_s})^{\perp_{g_\epsilon^{hk}} }$ and we write its right inverse as $P:L^2 \to W^{2,2}$. Notice also, that in the same way as $\| (J_{\gamma_0}^{g_\epsilon})^{-1} \|\lesssim \epsilon^{-1}$ we also must have $\|P \| \lesssim \epsilon^{-1}$. 
	
	Next, we search for a solution of the form $V=P(W)$. In this situation, the geodesic equation becomes 
	\begin{align*}
		0 & = (\nabla^{ g_\epsilon^{hk} }_{\dot{\gamma}_0} \dot{\gamma_0})^{ g_\epsilon^{hk} } +  s \frac{d\psi_s}{ds}\Big\vert_{s=0}\left( (\nabla^{g_\epsilon^{hk}}_{\dot{\gamma}_0} \dot{\gamma_0})^{\perp_{g_\epsilon^{hk}} }  \right)  + s W + \frac{s^2}{2} R_s(P(W)) ,
	\end{align*}
	which can be written as a fixed point equation
	\begin{equation}\label{eq:Fixed Point Equation}
		W = - \frac{ (\nabla^{g_\epsilon^{hk}}_{\dot{\gamma}_0} \dot{\gamma_0})^{\perp_{g_\epsilon^{hk}} } }{s} - \frac{d\psi_s}{ds}\Big\vert_{s=0}\left( (\nabla^{g_\epsilon^{hk}}_{\dot{\gamma}_0} \dot{\gamma_0})^{\perp_{g_\epsilon^{hk}} }  \right) - \frac{s}{2} R_s(P(W)) .
	\end{equation}
	We shall solve this using Schauder's fixed point theorem. For this we must show the the map $\cF:L^2 \to L^2$ given by the right hand side of the previous equation maps a ball of radius $R$ to itself. Let us prove this using equation \ref{eq:estimate on the remainder}
	\begin{align*}
		\| \cF (W) \|_{L^2} & \lesssim \frac{1}{s} \| (\nabla^{g_\epsilon^{hk}}_{\dot{\gamma}_0} \dot{\gamma_0})^{\perp_{g_\epsilon^{hk}} } \|_{L^2} + \Big\Vert \frac{d\psi_s}{ds}\Big\vert_{s=0}\left( (\nabla^{g_\epsilon^{hk}}_{\dot{\gamma}_0} \dot{\gamma_0})^{\perp_{g_\epsilon^{hk}} }  \right) \Big\Vert_{L^2} + s \| R_s(P(W)) \|_{L^2} \\
		& \lesssim \frac{\epsilon^{\alpha + \frac{1}{2}}}{s}  + s \| P(W) \|^2_{W^{2,2}} \\
		& \lesssim \frac{\epsilon^{\alpha + \frac{1}{2}}}{s}  + s \|P\|^2 R^2 .
	\end{align*}
	We shall then chose $\epsilon$ and $s$ so that both $\frac{\epsilon^{\alpha+ \frac{1}{2}}}{s} $ and $s \|P\|^2 R^2$ are smaller than $R/2$. Using the fact that $\|P\|\sim \epsilon^{-1}$, this will be the case if $sR<\epsilon^2/2$ and $\epsilon^{\alpha+ \frac{1}{2}} < sR/2$ which can clearly be achieved for sufficiently small $s$ and $\epsilon$ because $\alpha= \frac{11-2\delta}{5} >3/2$ and so $\alpha + \frac{1}{2}> 2$.
	
	Another possibility, which guarantees the uniqueness of the fixed point, is to use the contraction mapping principle. For this, we can use equation \ref{eq:estimate on the remainder (difference)} in the Appendix \ref{appendix:geodesic equation remainder}, which yields 
	\begin{align*}
		\| \cF (W_2) - \cF(W_1) \| & \lesssim  s \| P \|^2 R  \| W_2-W_1 \| ,
	\end{align*}
	and we conclude that $\cF$ is a contraction in the ball of radius $R$ if $s\|P\|^2 R <1$ which, up to a factor of two, coincides with a condition from the previous computation.
	Now, having established the existence of a solution $W$, with $\| W \| \leq R$ to the fixed point equation \ref{eq:Fixed Point Equation}, the geodesic it gives rise to is $\gamma_s(t)=\exp_{\gamma_0(t)}(sV(t))$ with $V=P(W)$. Hence, in order to estimate how large $sV$ is, we must find the smallest value of $s\|P\|R$ for which the previous argument can be made to work. Recall that the two conditions needed are
	$$sR < \frac{\epsilon^2}{2} , \ \text{and} \ \epsilon^{\alpha+\frac{1}{2}} < \frac{sR}{2}.$$
	Recall however that $\alpha =\frac{11-2\delta}{5}$ for some $\delta \in (-1/2,0)$. Hence, $\alpha+\frac{1}{2}= \frac{27-4\delta}{10}$ which can be made larger than $14/5$ by picking $\delta \in (-1/2,-1/4)$ and so we it is enough if $sR = C\epsilon^{\frac{14}{5}}$ for some uniform constant $C>0$ not depending on $\epsilon$. Notice that this allows one to shrink $s$ as much as one wants at the expense of increasing $R$ in a proportional manner so that their product remains $C\epsilon^{\frac{14}{5}}$. Hence, we find that 
	$$\|sV\|_{W^{2,2}} = \|sP(W)\|_{W^{2,2}} \lesssim s \| P \| \|W \|_{L^2} \lesssim sR \epsilon^{-1} = C \epsilon^{\frac{9}{5}},$$
	This finishes the proof of the first part of Proposition \ref{prop:solving the geodesic equation} using $U=sV$.
	
	Next, we must prove the claim regarding the Hausdorff distance of $\gamma_s$ to $\gamma_0$. The estimates we have obtained imply that
	$$\int_{0}^{2\pi \epsilon} \left( |sV(t)|^2  + |s\partial_t V (t)|^2 + |s\partial_t^2 V(t)|^2 \right) dt \lesssim \epsilon^{\frac{18}{5}}. $$
	Reparametrizing with a parameter $u=\epsilon^{-1}t$ which is independent of $\epsilon$ and defining $\tilde{V}(u)=V(\epsilon u)$, we can write this as
	$$\int_{0}^{2\pi} \left( |s \tilde{V}(u)|^2 + \epsilon^{-2} |s \partial_u \tilde{V}(u)|^2 + \epsilon^{-4} |s \partial_u^2 \tilde{V}(u)|^2 \right) \epsilon du \lesssim \epsilon^{\frac{18}{5}}. $$
	Given that $1 \ll \epsilon^{-1}$ we find that
	$$\int_{0}^{2\pi} \left( |s \tilde{V}(u)|^2 + |s \partial_u \tilde{V}(u)|^2 + |s \partial_u^2 \tilde{V}(u)|^2 \right)  du \lesssim \epsilon^{\frac{13}{5}}. $$
	The Sobolev embedding $W^{2,2} \hookrightarrow C^{1,\beta}$, for $\beta \in (0,1/2)$ in dimension $1$, implies in particular that
	$$\| sV \|_{C^0} \lesssim \epsilon^{\frac{13}{10}}.$$
	From this, we can give an upper bound on the Hausdorff distance $d_H$ between $\gamma_0(t)$ and $\gamma_s(t)=\exp_{\gamma_0(t)}(sV(t))$. This gives,
	\begin{align*}
		d_H (\gamma_0 , \gamma_s) & = \sup_{t} \mathrm{dist} (\gamma_0, \gamma_s(t)) \\
		& = \sup_{t} \inf_{t'} \mathrm{dist} (\gamma_0(t'), \gamma_s(t))  \\
		& \leq \sup_{t} \mathrm{dist} (\gamma_0(t), \gamma_s(t)) \\
		& = \sup_{t} |sV(t)| \\
		& \lesssim \epsilon^{13/10},
	\end{align*}
	which ends the proof of the second bullet.
	
	Finally, we turn to the last statement regarding the index. For this we shall fix a coordinate system as in the beginning of this section, i.e. $\lbrace x_\mu \rbrace_{\mu=1}^4$ satisfying $|\partial_{x_\mu}|_{g_\epsilon^{hk}} \lesssim 1$ and recall from equation \ref{eq:Christoffel symbols} that for $\alpha = \frac{11-2\delta}{5}$
	$$(\Gamma^{hk}_\epsilon)^\nu_{\mu \lambda} = (\Gamma_\epsilon)^\nu_{\mu \lambda} + \epsilon^\alpha (\delta \Gamma_\epsilon)^\nu_{\mu \lambda}, 
	$$
	with all derivatives. Then, in these coordinates, we can write the Jacobi operator $J_{\gamma_s}^{g^{hk}_\epsilon}$ of the geodesic $\gamma_s(t)$ with respect to $g_\epsilon^{hk}$ using the computation carried out in the Appendix \ref{appendix:geodesic equation remainder}. This gives, for a vector field $v=v^{\nu} \partial_\nu$,
	$$J_{\gamma_s}^{g^{hk}_\epsilon}(v) = (j^{hk}_\epsilon)^{\nu}(v) \partial_\nu,$$
	where
	\begin{align*}
		(j^{hk}_\epsilon)^{\nu}(v) & = (1-(g^{hk}_\epsilon)_{\rho \sigma} \dot{\gamma}_s^\sigma \dot{\gamma}_s^\nu) \left( \ddot{v}^\nu + \partial_\lambda (\Gamma^{hk}_\epsilon)^{\nu}_{\sigma \mu} \dot{\gamma}_s^\sigma \dot{\gamma}_s^\mu v^\lambda + 2 (\Gamma^{hk}_\epsilon)^{\nu}_{\mu \lambda} \dot{\gamma}_s^\mu \dot{v}^\lambda \right) .
	\end{align*}
	Now, using the previously mentioned relation between $g^{hk}_\epsilon$ (respectively $(\Gamma^{hk}_\epsilon)^\nu_{\mu \lambda}$) and $g_\epsilon$ (respectively $(\Gamma_\epsilon)^\nu_{\mu \lambda}$), we find that
	\begin{align*}
		(j^{hk}_\epsilon)^{\nu}(v) & = (1-(g_\epsilon)_{\rho \sigma} \dot{\gamma}_s^\sigma \dot{\gamma}_s^\nu) \left( \ddot{v}^\nu + \partial_\lambda (\Gamma_\epsilon)^{\nu}_{\sigma \mu} \dot{\gamma}_s^\sigma \dot{\gamma}_s^\mu v^\lambda + 2 (\Gamma_\epsilon)^{\nu}_{\mu \lambda} \dot{\gamma}_s^\mu \dot{v}^\lambda \right) + O(\epsilon^\alpha |v|).
	\end{align*}
	Furthermore, $\gamma_s(t)=\exp_{\gamma(t)}(sV(t))$ with $|sV|=O(\epsilon^{\frac{13}{10}})$ as just proven. Hence, we find that $\gamma_s^{mu}(t)=\gamma_0^{\mu}(t) + O(\epsilon^{\frac{13}{10}})$ and inserting this into above combined with the fact that $\epsilon^\alpha < \epsilon^{\frac{13}{10}}$ yields
	\begin{align*}
		(j^{hk}_\epsilon)^{\nu}(v) & = (1-(g_\epsilon)_{\rho \sigma} \dot{\gamma}_0^\sigma \dot{\gamma}_0^\nu) \left( \ddot{v}^\nu + \partial_\lambda (\Gamma_\epsilon)^{\nu}_{\sigma \mu} \dot{\gamma}_0^\sigma \dot{\gamma}_0^\mu v^\lambda + 2 (\Gamma_\epsilon)^{\nu}_{\mu \lambda} \dot{\gamma}_0^\mu \dot{v}^\lambda \right) + O(\epsilon^{\frac{13}{10}} |v|) \\
		& = j_\epsilon(v) + O(\epsilon^{\frac{13}{10}} |v|) ,
	\end{align*}
	where $j_\epsilon(v)$ is defined by $J_{\gamma_0}(v)=(j_\epsilon)^\nu(v) \partial_\mu$. We have proved in Proposition \ref{prop:Jacobi eigenvalues 2} that the lower absolute value of an eigenvalue of $J_{\gamma_0}$ is proportional to $\epsilon$ and as $\epsilon^{\frac{13}{10}} \ll \epsilon$ the indexes of $J_{\gamma_s}^{g^{hk}_\epsilon}$ and $J_{\gamma_0}$ coincide.
	\end{proof}

	\subsection{Critical points of $h$}\label{ss:Critical points of h}
	
	The next proposition gives a lower bound on the number of critical points of $h$ for the generic position of the points $p_1, \ldots , p_n \in \mathbb{T}_\Lambda$. Recall that these critical points give rise to geodesics for $K_\epsilon^{gh}$. 
	
	\begin{proposition}\label{prop:critical points}
		Let $h$ be a solution to \ref{eq:Poisson equation} for a generic arrangement of points $p_1, \ldots , p_n$. Then, $h$ contains at least $10$ critical points of index one and $2(n+1)$ critical points of index two.  
		
		In particular, viewing $h$ as a function in the quotient $T/\mathbb{Z}_2$, it has at least $5$ critical points of index $1$ and $n+1$ critical points of index $2$.
	\end{proposition}

	\begin{proof}
	
	Let $h$ be a solution to \ref{eq:Poisson equation} for a generic arrangement of points $p_1, \ldots , p_n$. It follows from Theorem 6.3 in \cite{Morse} and its second application in the middle of page 46 in the same reference, that for the generic choice of a single point, say $p_1 \in \mathbb{T}^3$, the function $h$ is Morse. Hence, we can now use Morse theory in order to compute a lower bound on the number of critical points of $h$. We start by noticing that it is possible to modify $h$ near the points $\lbrace q_1 , \ldots , q_8 \rbrace$, and $ \lbrace p_1, \ldots , p_n , \tau(p_1) , \ldots , \tau(p_n) \rbrace$ so that it extends to a smooth function $\tilde{h}$ on $\mathbb{T}^3$ with a maximum at each $p_i$ and $\tau(p_i)$ and a minimum at each $q_i$. This can be done in a manner so as to not introduce any more critical points of $h$. 
	
	Now, let $b_i$ denote the Betti numbers of $\mathbb{T}^3$ and $P_t=\sum_l b_l t^l$ its Poincar\'e polynomial. Similarly, let $\nu_l(h)$ be the number of index $l$ critical point of $h$ and $M_t(f)=\sum_l \nu_l(f) t^l$ be its Morse function. The strong Morse inequalities are then given by
	$$M_t(h)-P_t=(1+t) Q_t(h),$$
	where $Q_t(h)= \sum_l a_l t^l$ is a polynomial all of whose coefficients are nonnegative. In our situation we have $P_t=1+3t+3t^2+t^3$ and $\nu_0=8$, $\nu_3=2n$ and inserting in the strong Morse inequalities above we find
	\begin{align*}
		& 8 + \nu_1 t + \nu_2 t^2 + 2n t^3 - (1+3t+3t^2+t^3)  = (1+t) (a_0+a_1t +a_2 t^2 +a_3 t^3 + \ldots ) \Leftrightarrow \\
		\Leftrightarrow & 7 + (\nu_1-3) t + (\nu_2 -3) t^2 + (2n-1) t^3  = a_0 + (a_0+a_1)t + (a_1+a_2) t^2 + (a_2 +a_3) t^3 + (a_3+a_4) t^4 + \ldots .
	\end{align*}
	Then, we find that
	\begin{align*}
		a_0 & = 7 \\
		a_0+a_1 & = \nu_1-3 \\
		a_1+a_2 & = \nu_2-3 \\
		a_2+a_3 & = 2n-1 \\
		a_3+a_4 & = 0.
	\end{align*}
	From the last of these we find that $a_3=0=a_4$, then from the semilast and the first we respectively have $a_2=2n-1$ and $a_0=7$. We can then insert these into the remaining equations and solve them for $\nu_1$ and $\nu_2$. The result is
	\begin{align*}
		\nu_1 & = 10+a_1 \geq 10 \\
		\nu_2 & = 2(n+1)+a_1 \geq 2(n+1),
	\end{align*}
	and we conclude that $h$ has at least $10+2(n+1)$ critical points. These will never be placed at the fixed points of the $\mathbb{Z}_2$-action as those correspond to the points $\lbrace q_1 , \ldots , q_8 \rbrace$ at which $\tilde{h}$ has local minima. Hence, as $h$ is $\mathbb{Z}_2$-invariant, down in $T/\mathbb{Z}_2$ this corresponds to $5+n+1=n+6$ critical points, $5$ of which have index one and $n+1$ having index two.
	\end{proof}

	\subsection{A particularly interesting example}\label{ss:cube}
	
	Before jumping into the relevant example we shall start with a little baby example in one lower dimension.
	
	\begin{example}\label{ex:square}
		Consider $\mathbb{T}^2=\mathbb{R}^2/\mathbb{Z}^2$ and a function $\phi$ on $\mathbb{T}^2$ which has $4$ local minima at the points of $(\mathbb{Z}/2)^2/\mathbb{Z}^2$ and $4$ local maxima at the points 
		$$p_1 = (1/4,0) , \ p_2 = (3/4,0) , \ p_3 = (0,1/4) , \ p_4=(0,3/4)  \in \mathbb{T}^2,$$
		as illustrated in figure \ref{fig:square}. If $\phi$ is a Morse function and these are its only local maxima and minima, then it has $8-\chi(T^2)=8$ critical points of index $1$. Suppose now that $\phi$ is invariant by reflection on the lines $y=x$ and $x+y=1$, then we claim that $4$ of the $8$ critical points are located along these lines. Indeed, by the corresponding reflection symmetry, along any such line $\nabla \phi$ must be parallel to it. Any such line corresponds to a circle in $\mathbb{T}^2$ and $\phi$ has two local maximums along a line as any such intersects $(\mathbb{Z}/2)^2/\mathbb{Z}^2$ in two points. Hence, $\phi$ must have two local minima in any if these lines. All together these give rise to $4$ critical points. Again, see figure \ref{fig:square} and the caption below it which explain this example in a more visual manner.
		
		\begin{figure}[h]\label{fig:square}
			\centering
			\includegraphics[width=0.65\textwidth,height=0.38\textheight]{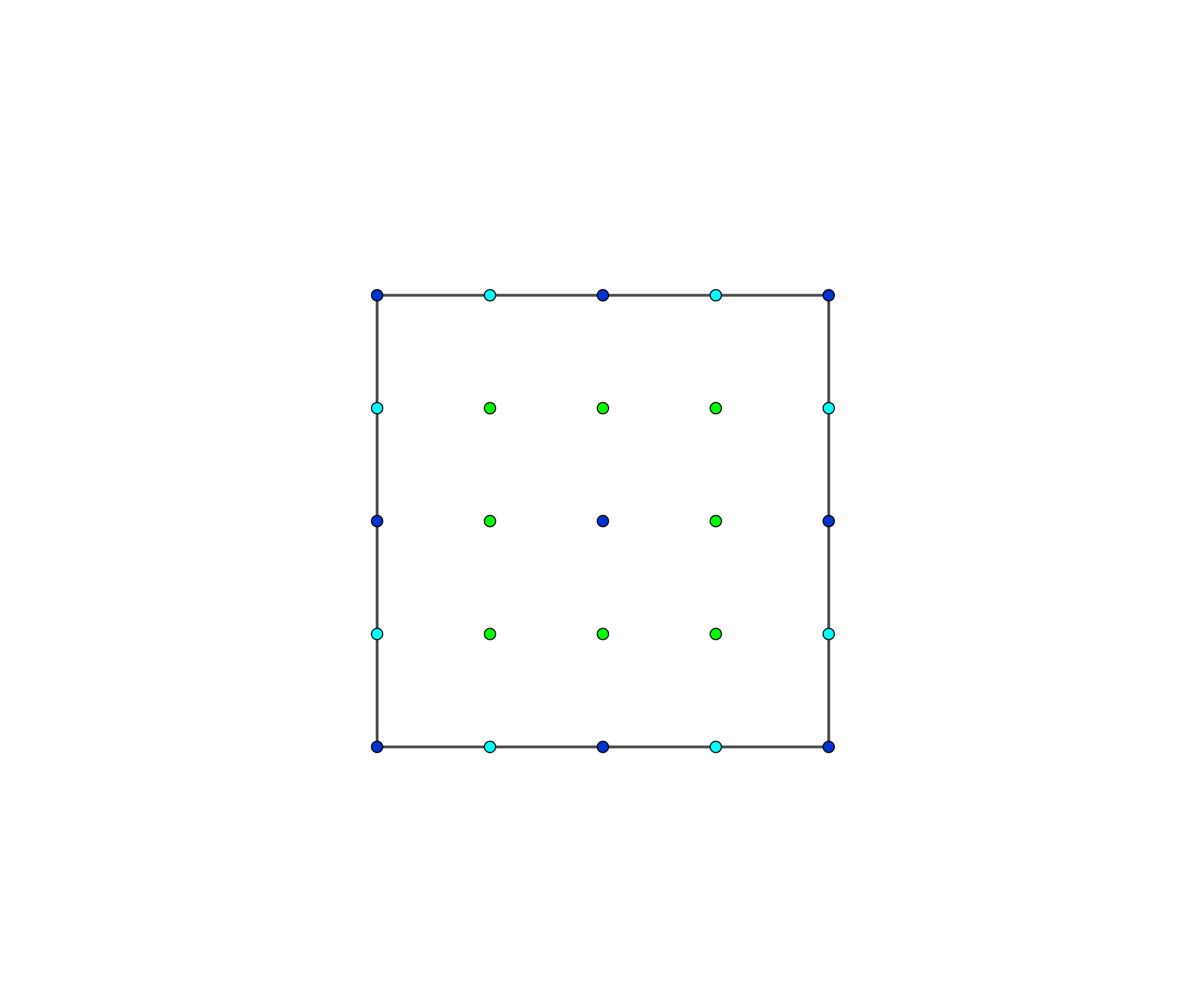}
			\caption{Fundamental domain for $\mathbb{T}^2$ with local maximuns at the blue points and local minima at the black points. The green points are the extra critical points which are local minima of $\phi$ when restricted to lines corresponding to the fixed locus of reflexion symmetries.}
		\end{figure}
	
	The remaining $4$ critical points can be obtained by a similar reasoning by reflection on the lines $y=1/4$ and $y=3/4$.
		
	\end{example}
	
	Let $\Lambda \subset \mathbb{R}^3$ be a full rank lattice and $\mathbb{T}_\Lambda=\mathbb{R}^3/\Lambda$ the corresponding $3$-torus. We shall consider the function $h$ which solves \ref{eq:Poisson equation} with $\lbrace q_1 , \ldots , q_8 \rbrace= (\Lambda/2)/\Lambda$ the fixed points of $\tau$, $n=4$, $k_1 = \ldots = k_4 =4$, and $\lbrace p_1, \ldots p_n, \tau(p_1) , \ldots , \tau(p_n) \rbrace $ chosen in the way which we shall now explain. Explicitly write the lattice $\Lambda$ as follows
	$$\Lambda= \lbrace n_1 \lambda_1 + n_2 \lambda_2 + n_3 \lambda_3 \ | \ (n_1,n_2,n_3) \in \mathbb{Z}^3 \rbrace . $$
	Then, we set
	$$p_1  = \lambda_1/4 , \ p_2  = 3\lambda_1/4 , \ p_3 = \lambda_2/4 , \ p_4 = 3\lambda_2/4 ,$$
	then
	$$-p_1  = \lambda_1/4 +\lambda_3/2 , \ - p_2  = 3\lambda_1/4 +\lambda_3/2 , \ -p_3 = \lambda_2/4 +\lambda_3/2 , \ -p_4 = 3\lambda_2/4 + \lambda_3/2 .$$
	This can be thought of as stacking two copies of example \ref{ex:square} in $\mathbb{T}_\Lambda$, see figure \ref{fig:cube} below where a fundamental domain is pictured and keep in mind that the opposite sides are being identified.
	
	\begin{figure}[h]\label{fig:cube}
		\centering
		\includegraphics[width=0.65\textwidth,height=0.35\textheight]{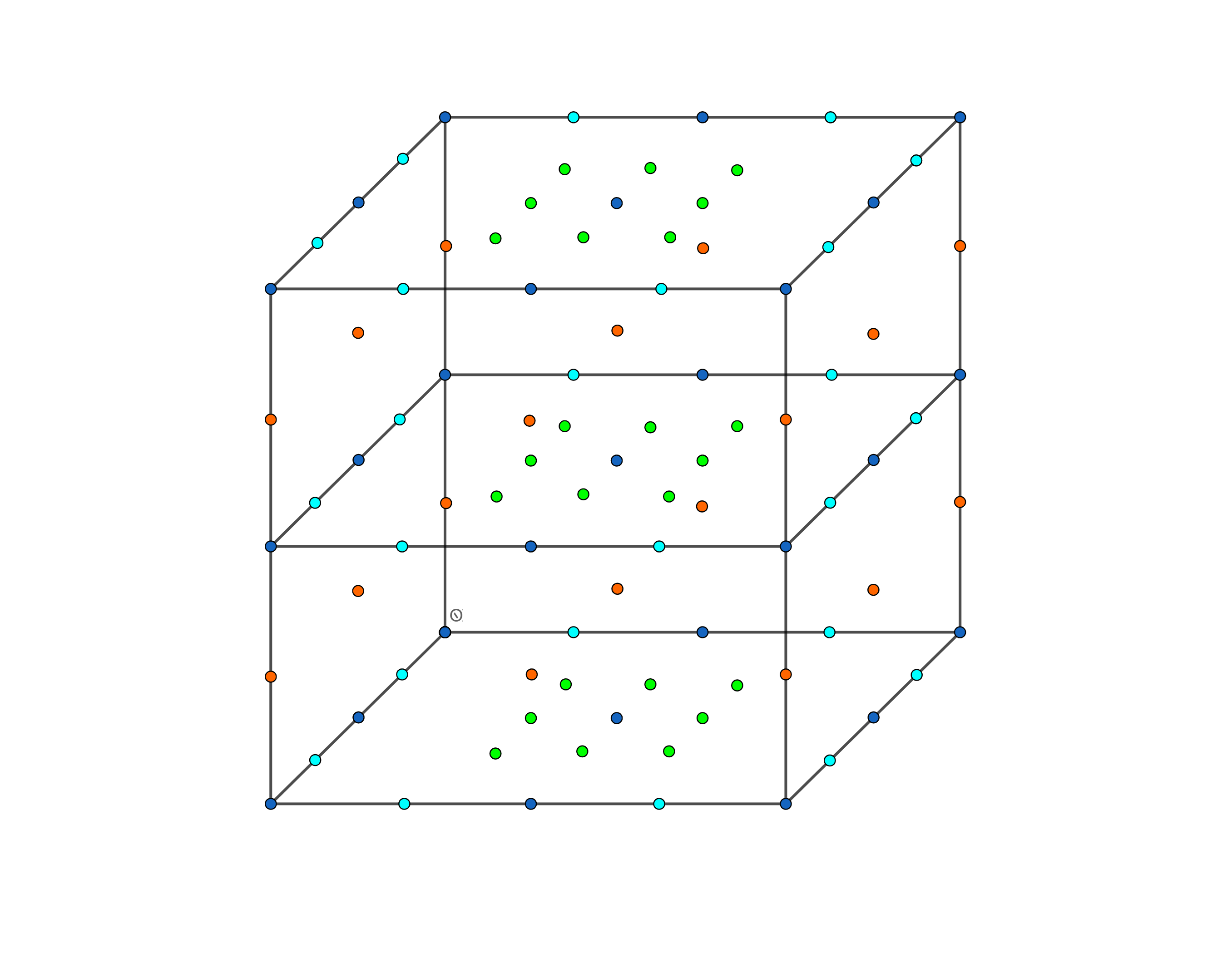}
		\caption{Fundamental domain for $\mathbb{T}_\Lambda$. The $8$ black points represent charges of magnitude $-4$ and the $8$ blue dots represent charges of magnitude $+4$. Then, the green points represent the $8$ critical points which we obtain from the two-dimensional picture and by reflection symmetry on the plane containing them. Furthermore, we have $8$ extra critical points, in orange, each of which can be obtained by a similar manner by reflection symmetry argument.}
	\end{figure}

	\section{Geodesics on Foscolo's K3 surface}\label{sec:examples}

	\subsection{The general result, Theorem \ref{thm:Main}}
	
	For convenience, we recall here the statement of Theorem \ref{thm:Main}.
	
	\begin{theorem}\label{thm:Main extended}
		Let $n \geq 1$, $k_1, \ldots , k_n, m_1, \ldots , m_8$ satisfying \ref{eq:necessary condition for Laplace equation}, a collection of ``bubbles'' $N^1, \ldots , N^n$, $M^1, \ldots , M^8$, and $p_1, \ldots , p_n \in \mathbb{T}_\Lambda$ with $p_1$ generically chosen. Then, there is a positive number $\epsilon_0 \ll 1$, such that for all $\epsilon < \epsilon_0$: 
		\begin{itemize}
			\item There are at least $n+1$ closed geodesics of index one and $5$ closed geodesics of index two in $K_\epsilon$. These can be located within a precision of $\epsilon^{13/10}$, in the Hausdorff distance, as the circle fibers above the critical points of $h$. 
		
			\item For $i=1, \ldots , n$ and if the location of $x^{i}_1$ in the ``bubble'' $N^i$ is generic, there are at least $k_i-1$ closed geodesics on $N_\epsilon^i$. In addition, there is a non-negative continuous and monotone function $\kappa: [0,\epsilon_0] \to \mathbb{R}$ with $\kappa(0)=0$ such that these geodesics are within Hausdorff distance $\epsilon \kappa(\epsilon)$ from the circle fibers above critical points of $\phi_i$. Furthermore, there is $d_{k_i} >0$ such that if $\min_{j_1 \neq j_2}|x^{i}_{j_1} - x^{i}_{j_2}| > d_{k_i}$, then these geodesics have index $1$. 
		\end{itemize}
	\end{theorem}

	\begin{proof}
	The proof of the first part follows from deforming the non-degenerate geodesics on $K_\epsilon^{gh}$. The existence of these, with the required index, for the approximate metric $g^{gh}_\epsilon$ follows from Proposition \ref{prop:critical points} which guarantees the existence of at least $n+1$ index one and $5$ index two closed geodesics on this region. By Proposition \ref{prop:solving the geodesic equation}, these can then be deformed to geodesics for the metric $g_\epsilon^{hk}$ within Hausdorff distance $\epsilon^{13/10}$ of the original geodesics for $g^{gh}_\epsilon$, which recall, are located above the critical points of $h$.
	
	The second part of the theorem follows from applying White's deformation theorem as in section \ref{sec:deform geodesics concentrated}. Indeed, for each $i=1, \ldots , n$, the restriction of $\epsilon^{-2}g_\epsilon^{hk}$ to $N_\epsilon^i$ converges to an ALF gravitational instanton of type $A_{k_i-1}$ which we denote by $N_0^i$. Under the hypothesis that $x^{i}_1$ is generic and $|x^{i}_{j_1} - x^{i}_{j_2}| > d_{k_i}$, we can apply Theorem \ref{thm:GH ALF A_k} (see also remark \ref{rem:rescaled statement 0}), which ensures the existence of at least $k_i-1$ nondegenerate closed geodesics of index one. These persist as closed geodesics of index one for $g^{hk}_\epsilon$ by White's deformation theorem \cite{White}.
	\end{proof}

	\subsection{Examples}

	Let us see a few examples when all the $m_j$'s vanish. In this situation, the condition in equation \ref{eq:necessary condition for Laplace equation} reads
	$$\sum_{i=1}^n k_i = 16,$$
	and we shall now explore a few examples from different partitions of $16$.
	
	\begin{example}\label{ex:all k=1}
		If all $k_i=1$ then we must have $n=16$ points and so, by Theorem \ref{thm:Main extended}, there are at least $5$ closed geodesics of index two and $17$ of index one in the region $K_\epsilon$.
		
		Alternatively, we can carry out the arguments used to prove Theorem \ref{thm:Main extended} in more extent. For this, we use Proposition \ref{prop:critical points}, which guarantees that for the generic location of one of the singularities $h$ has at least $22$ non-degenerate critical points. Of these $5$ have index one and $17$ have index two and by Proposition \ref{prop:Jacobi eigenvalues} this results in $17$ closed geodesics of index one and $5$ closed geodesics of index two for $g_\epsilon$. Then, it follows from Proposition \ref{prop:solving the geodesic equation} that such geodesics deform to closed geodesics for the hyperK\"ahler metric $g_\epsilon^{hk}$. Furthermore, each $N^i_0$ space that is being glued in is a Taub-Nut space and it admits no closed geodesics, see example B.1 in the Appendix B of \cite{Lotay}. 
		
		Furthermore, around each of the $8$ points $q_j$, the metric $\epsilon^{-2} g_\epsilon^{hk}$ approximates an ALF $D_0$-manifold, namely the Atiyah--Hitchin manifold. It then follows from the discussion in section \ref{ss:geodesics on AH} that in each $M^{\epsilon}_j$  there are at least $3$ closed geodesics which given an extra $24$ closed geodesics on $X_\epsilon$. Thus, bringing the total to $22+24=46$ closed geodesics. Notice however that the index of these last $24$ geodesics is unknown.
	\end{example} 

	\begin{example}
		If all $k_i=2$, then $n=8$ and by Theorem \ref{thm:Main extended} we find that in $K_\epsilon$ there are $5$ closed geodesics of index two and $9$ of index one. Again, we can alternatively start arguing by applying Proposition \ref{prop:critical points} for the generic location of the singularities, which yields $14$ critical points of $h$. Of these, $5$ have index one and $9$ index two. Resulting, by Proposition \ref{prop:Jacobi eigenvalues}, in closed geodesics for $g_\epsilon$, $5$ of which have index two and $9$ have index one. Again, by Proposition \ref{prop:solving the geodesic equation}, these can be deformed to closed geodesics with respect to the hyperK\"ahler metric $g_\epsilon^{hk}$ in $K_\epsilon$.
		
		Notice however that, in this case $\epsilon^{-2} g_\epsilon^{hk}$ approximates an ALF $A_1$ manifold around each point $p_i$. As explained in example \ref{ex:A_1}, each $N^i_\epsilon$ admits at least $3$ closed geodesics with $2$ having index one and $1$ having index two. Thus, resulting in a total of $24$ extra closed geodesics $8$ of each index one, two, and three. Adding to the geodesics we had previously found in $K_\epsilon$ we have so far counted $8+9=17$ closed geodesics of index one, $8+5=13$ of index two, and $8$ of index three.
		
		Furthermore, just like in the previous example, around each of the $8$ points $q_j$, $\epsilon^{-2} g_\epsilon^{hk}$ converges to an ALF $D_0$-manifold, which gives rise to extra $8\times 3=24$ closed geodesics on $X_\epsilon$ of unknown index. Thus, in total, we find $38+24=62$ closed geodesics. 
	\end{example}
	 
	\begin{example}
		Let us now consider a different example where we have $n_1$ singularities with $k_i=2$ and $n_2$ singularities with $k_i$ equal to $3$. Such $n_1$ and $n_2$ must satisfy $32 =2 \sum_{i=1}^{n_1} 2 + 2 \sum_{i=1}^{n_2} 3$, i.e.
		$$16= 2n_1 + 3n_2.$$
		
		The only possible solution to this equation for $(n_1,n_2) \in \mathbb{N}_0 \times \mathbb{N}$ is with $n_1=2$ and $n_2=4$. Then, $h$ has singularities at $p_1,p_2$ with charge $2$ and at $p_3, \ldots , p_6$ with charge $3$. Around these points $\epsilon^{-2} g_\epsilon^{hk}$ converges to $A_1$ and $A_2$ ALF gravitational instanton respectively. The later of these ALF gravitational instantons can be chosen so that they are constructed from the Gibbons Hawking ansatz with $3$ singularities placed at the vertices of an equilateral triangle as in Proposition \ref{prop:triangle} and example \ref{ex:A_2}. Then, there are $4$ non-degenerate closed geodesics ($3$ of index one and $1$ of index two) in each of the $A_2$ spaces and a totally geodesic ellipsoid on the $A_1$ ones. Using a similar reasoning to that of the previous example, these give rise to a total of $3 n_1 + 4 n_2=6+16=22$ closed geodesics on $\bigcup_{i=1}^n N^i_\epsilon$ with $2+4\times 3=14$ having index one, $2+4\times 1=6$ index two, and $2$ index three. 
		
		Together with the geodesics that we previously found in $K_\epsilon$, we obtain $22+n+6=22+6+6=34$ closed geodesics on $K_\epsilon \cup \bigcup_{i=1}^n N^i_\epsilon$. Of these, there are $14+7=21$ of index one, $6+5=11$ of index two, and $2$ having index three.
		
		Furthermore, as in all other examples so far, adding the $8$ ALF instantons of type $D_0$ yields an extra $8\times 3= 24$ closed geodesics, resulting in a total of $34+24=58$.
	\end{example}

	\begin{example}
		Now, we consider the following partition $9+3+2+2$ of $16$, then $n=4$ and the hyperK\"ahler metric has $5$ closed geodesics of index one and $5$ closed geodesics of index two in $K_\epsilon$. 
		
		Furthermore, we consider the ALF gravitational instantons of type $A_8$ and $A_2$ of examples \ref{ex:A_8} and \ref{ex:A_2} respectively. The first of these contributes $12$ closed geodesics of index one and $4$ of index two, while the second contributes $3$ closed geodesics of index one, and $1$ of index two. Overall, in the region $\bigcup_{i=1}^4 N^i_\epsilon$ we obtain $12+3+2\times 1=17$ closed geodesics of index one, $4+1+2\times 1=7$ of index two, and $2\times 1$ of index three.
		
		All together, in $K_\epsilon \cup \bigcup_{i=1}^4 N^i_\epsilon$ we find $5+17=22$ closed geodesics of index one, $5+7=12$ of index two, and $2$ of index three.
		
		Finally, together with the $24$ closed geodesics from $\bigcup_{j=1}^8 M^i_\epsilon$, we have found $24+12+24=60$ closed geodesics.
	\end{example}

	\begin{example}
	In this example we consider instead the partition $9+3+3+1$, where again $n=4$ and so the hyperK\"ahler metric also has $5$ closed geodesics of index one and $5$ closed geodesics of index two in the large $K_\epsilon$ region. 
	
	Using the same $A_8$  and $A_2$ spaces as in the previous example, we find that, in the region $\bigcup_{i=1}^4 N^i_\epsilon$, we obtain $12+3 \times 2=18$ closed geodesics of index one and $4+2\times 1=6$ closed geodesics of index two.
	
	All together, in $K_\epsilon \cup \bigcup_{i=1}^4 N^i_\epsilon$ we find $5+18=23$ closed geodesics of index one, and $5+6=11$ closed geodesics of index two.
	
	Together with the $24$ closed geodesics from $\bigcup_{j=1}^8 M^i_\epsilon$, we have found a total of $23+11+24=58$ closed geodesics.
	\end{example}
	
	\begin{example}
		Finally, we consider the situation which lies at the other extreme of example \ref{ex:all k=1}. This is when $n=1$ in which case $h$ has $7$ critical points and $\epsilon^{-2} g_\epsilon^{hk}$ approximates an ALF $A_{15}$-manifold which, in the generic situation, always has at least $14$ closed geodesics. Adding these to the previous $7$ located in $K_\epsilon$ we obtain a total of $22$ closed geodesics.
	\end{example}

	\appendix

	\section{Deforming the geodesic equation}\label{appendix:geodesic equation remainder}
	
	Consider a family of curves $\gamma_s(t)=\exp_{\gamma_0(t)}(sV(t))$ which in a local coordinate system $(x^1, \ldots , x^n)$ can be written as $t \mapsto (\gamma_s^1(t) , \ldots , \gamma_s^n(t) ) $. The geodesic equation is
	\begin{align*}
		f_s(V) & = (\nabla_{\dot{\gamma}_s} \dot{\gamma}_s )^{\perp} = \underbrace{\nabla_{\dot{\gamma}_s} \dot{\gamma}_s}_{I_1(s)} - \underbrace{ \langle  \nabla_{\dot{\gamma}_s} \dot{\gamma}_s , \frac{\dot{\gamma}_s}{|\dot{\gamma}_s|} \rangle \frac{\dot{\gamma}_s}{|\dot{\gamma}_s|} }_{I_2(s)},
	\end{align*}
	and we can Taylor expand this in $s$ as
	\begin{equation}\label{eq:Appendix Taylor expansion of the equation}
		f_s(V) = f_0(V) + s \frac{d}{ds} \Big\vert_{s=0} f_s(V) + \frac{s^2}{2} R_s(V),
	\end{equation}
	where $|R_u(V)| = | \frac{1}{u^2} \int_0^u (u-s)^2 \partial_s^2 f_s(V) ds | \lesssim \sup_{s \in (0,u)} |\partial_s^2 f_s(V)|$. Now, in order to compute $\partial_s\vert_{s=0} f_s(V) =I_1'(0)+I_2'(0)$ and $\partial_s^2 f_s(V)=I_1''(s)+I_2''(s)$, we independently compute $I_1'(0)$, $I_2'(0)$ and $I_1''(s)$, $I_2''(s)$. For this, it is convenient to write the equation in local coordinates. The $\nu$-components of $I_1$ and $I_2$ are given by
	\begin{align*}
		I_1^{\nu}(s) & = \ddot{\gamma}_s^\nu + \Gamma^{\nu}_{\mu \lambda} (\gamma_s(t)) \dot{\gamma}_s^\mu \dot{\gamma}_s^\lambda , \\
		I_2^{\nu}(s) & = \frac{\dot{\gamma}_s^\nu g_{\rho \sigma}(\gamma_s(t)) \dot{\gamma}_s^\sigma I_1^\rho(s)}{g_{\rho \sigma}(\gamma_s(t)) \dot{\gamma}_s^\rho \dot{\gamma}_s^\sigma}  .
	\end{align*}
	With no loss of generality we can assume that $\gamma(t)$ is a geodesic parametrized by arclenght, i.e. $I_1(0)=0$. Furthermore, now we shall choose a very special set of coordinates. First we fix a coordinate $x_1$ along $\gamma_0(t_0-\epsilon, t_0+\epsilon )$ so that $x_1 \mapsto \exp(x_1 \dot{\gamma_0(t_0)}) = \gamma_0(t_0+x_1)$. Then, we fix an orthormal basis $\lbrace e_2 , \ldots , e_n \rbrace$ to the normal space to $\dot{\gamma_0}(t_0)$ at $\gamma_0(t_0)$ and which we parallel transport along $\gamma_0(t_0-\epsilon,t_0+\epsilon)$. Then, we consider the coordinate system
	$$(x_1, \ldots , x_n) \mapsto \exp_{\gamma_0(t_0+x_1)}(x_2 e_2 + \ldots +x_n e_n).$$
	In these coordinates we can write the vector fields $V(t)$ as $v_2(t) e_2 + \ldots + v_n(t) e_n$ and so the curves $\gamma_s(t)=\exp_{\gamma_0(t)}(sV(t))$ are given in these coordinates by
	$$t \mapsto (\gamma_s^1(t) , \ldots , \gamma_s^n(t) )=(t-t_0, sv_2(t) , \ldots , sv_n(t)).$$
	Then, we find that
	\begin{align*}
		(I_1^{\nu})'(0) & = \ddot{v}^\nu + \partial_\lambda \Gamma^{\nu}_{\sigma \mu} \dot{\gamma}^\sigma \dot{\gamma}^\mu v^\lambda + 2 \Gamma^{\nu}_{\mu \lambda} \dot{\gamma}^\mu \dot{v}^\lambda ,
	\end{align*}
	where $\dot{\gamma}^\sigma:=\dot{\gamma_0}^\sigma=\delta^{\sigma 1}$ but which we shall leave as $\dot{\gamma}^\sigma$ in order to more easily read the equations which we will find.
	
	The equation above can be written in a more friendly manner by noticing that along a geodesic, we have
	\begin{align*}
		\left( \frac{D^2V}{dt^2} \right)^\nu & = \ddot{v}^\nu+ 2 \Gamma^{\nu}_{\mu \lambda} \dot{\gamma}^\mu \dot{v}^\lambda + (\partial_\sigma \Gamma^\nu_{\mu \lambda} + \Gamma^{\nu}_{\mu \rho} \Gamma^{\rho}_{\sigma \lambda} -  \Gamma^{\nu}_{\rho \lambda}  \Gamma^{\rho}_{\mu \sigma} ) \dot{\gamma}^\sigma \dot{\gamma}^\mu v^\lambda ,
	\end{align*}
	and upon inserting above yields
	\begin{align*}
		(I_1^{\nu})'(0) & = \left( \frac{D^2V}{dt^2} \right)^\nu  + (  \partial_\lambda \Gamma^{\nu}_{\sigma \mu} - \partial_\sigma \Gamma^\nu_{\mu \lambda} +  \Gamma^{\nu}_{\rho \lambda}  \Gamma^{\rho}_{\mu \sigma} - \Gamma^{\nu}_{\mu \rho} \Gamma^{\rho}_{\sigma \lambda} ) \dot{\gamma}^\sigma \dot{\gamma}^\mu v^\lambda \\
		& =  \left( \frac{D^2V}{dt^2} \right)^\nu + R^\nu_{\mu \lambda \sigma} \dot{\gamma}^\sigma \dot{\gamma}^\mu v^\lambda .
	\end{align*}
	On the other hand, the second derivatives are given by
	\begin{align*}
		(I_1^{\nu})''(s) & = 2 v^\rho \partial_\rho \Gamma^{\nu}_{\mu \lambda} (\gamma_s(t)) \dot{\gamma}^\mu \dot{v}^\lambda + \frac{1}{2} v^\rho v^\sigma \partial^2_{\rho \sigma} \Gamma^\nu_{\mu \lambda}(\gamma_s(t)) \dot{\gamma}^\mu \dot{\gamma}^\lambda,
	\end{align*}
	which satisfies
	\begin{equation}\label{eq:bound on I_1'(0)}
		\sup_{s \in (0,u)}|(I_1^{\nu})''(s)| \lesssim |V|^2 + |\nabla V|^2 .
	\end{equation}
	We now turn to the computations of $I_2'(0)$ and $I_2''(s)$. Regarding the first of these, we have
	\begin{align*}
		|\gamma_s(t)|_g^2 (I_2^\nu)'(s) & = \left( v^\lambda \partial_\lambda g_{\rho \sigma}(\gamma_s(t)) \dot{\gamma}_s^\sigma \dot{\gamma}_s^\nu  + \dot{v}^\nu g_{\rho \sigma}(\gamma_s(t)) \dot{\gamma}_s^\sigma + \dot{\gamma}_s^\nu g_{\rho \sigma}(\gamma_s(t)) \dot{v}^\sigma   \right) I_1^\rho(s) \\
		& - \frac{1}{|\gamma_s(t)|_g^2}\frac{d}{ds} \left( g_{\rho \sigma}(\gamma_s(t)) \dot{\gamma}_s^\rho \dot{\gamma}_s^\sigma \right) \dot{\gamma}_s^\nu g_{\rho \sigma} \dot{\gamma}_s^\sigma    I_1^\rho(s) \\
		& + g_{\rho \sigma}(\gamma_s(t)) \dot{\gamma}_s^\sigma \dot{\gamma}_s^\nu (I_1^\rho)'(s) ,
	\end{align*}
	where
	$$\frac{d}{ds} \left( g_{\rho \sigma}(\gamma_s(t)) \dot{\gamma}_s^\rho \dot{\gamma}_s^\sigma \right) = v^\lambda \partial_\lambda g_{\rho \sigma}(\gamma_s(t)) \dot{\gamma_s}^\rho \dot{\gamma_s}^\sigma + g_{\rho \sigma} (\gamma_s(t)) (\dot{v}^\rho \dot{\gamma_s}^\sigma + \dot{\gamma_s}^\rho \dot{v}^\sigma), $$
	and using the fact that $I_1(0)=0$, we have
	\begin{align*}
		(I_2^\nu)'(0)  
		& = g_{\rho \sigma} \dot{\gamma}^\sigma \dot{\gamma}^\nu (I_1^\rho)'(0) .
	\end{align*}
	Next, we compute the second derivatives as follows
	\begin{align*}
		|\dot{\gamma}_s|^2(I_2^\nu)''(s) & = -\partial_s (|\dot{\gamma}_s|^2)(I_2^\nu)'(s) + \partial_s \left( |\dot{\gamma}_s|^2 (I_2^\nu)'(s) \right),
	\end{align*}
	and the second of these terms is given by
	\begin{align*}
		\partial_s \left( |\dot{\gamma}_s|^2 (I_2^\nu)'(s) \right) & = \left[ v^{\lambda_1} v^{\lambda_2} \partial^2_{\lambda_1 \lambda_2} g_{\rho \sigma}(\gamma_s(t)) \dot{\gamma}_s^\sigma \dot{\gamma}_s^\nu + v^\lambda \partial_\lambda g_{\rho \sigma}(\gamma_s(t)) \dot{v}^\sigma \dot{\gamma}_s^\nu + v^\lambda \partial_\lambda g_{\rho \sigma}(\gamma_s(t)) \dot{\gamma}_s^\sigma \dot{v}^\nu \right] I_1^\rho(s) \\
		& + \left[ v^\nu v^\lambda \partial_\lambda g_{\rho \sigma}(\gamma_s(t)) \dot{\gamma}_s^\sigma + 2 \dot{v}^\nu g_{\rho \sigma}(\gamma_s(t)) \dot{v}^\sigma + \dot{\gamma}_s^\nu v^\lambda \partial_\lambda g_{\rho \sigma}(\gamma_s(t)) \dot{v}^\sigma  \right] I_1^\rho(s) \\
		& + \left( v^\lambda \partial_\lambda g_{\rho \sigma}(\gamma_s(t)) \dot{\gamma}_s^\sigma \dot{\gamma}_s^\nu  + \dot{v}^\nu g_{\rho \sigma}(\gamma_s(t)) \dot{\gamma}_s^\sigma + \dot{\gamma}_s^\nu g_{\rho \sigma}(\gamma_s(t)) \dot{v}^\sigma   \right) (I_1^\rho)'(s) \\
		& + \frac{1}{|\gamma_s(t)|_g^2} \left[ \frac{1}{|\gamma_s(t)|_g^2} \left( \frac{d}{ds} \left( g_{\rho \sigma}(\gamma_s(t)) \dot{\gamma}_s^\rho \dot{\gamma}_s^\sigma \right) \right)^2  - \frac{d^2}{ds^2} \left( g_{\rho \sigma}(\gamma_s(t)) \dot{\gamma}_s^\rho \dot{\gamma}_s^\sigma \right) \right] \dot{\gamma}_s^\nu g_{\rho \sigma} \dot{\gamma}_s^\sigma    I_1^\rho(s) \\
		& - \frac{\frac{d}{ds} \left( g_{\rho \sigma}(\gamma_s(t)) \dot{\gamma}_s^\rho \dot{\gamma}_s^\sigma \right)}{|\gamma_s(t)|_g^2} \left[ ( \dot{v}^\nu g_{\rho \sigma} \dot{\gamma}_s^\sigma  + \dot{\gamma}_s^\nu v^\lambda \partial_\lambda g_{\rho \sigma} \dot{\gamma}_s^\sigma + \dot{\gamma}_s^\nu g_{\rho \sigma} \dot{v}^\sigma  ) I_1^\rho(s) + \dot{\gamma}_s^\nu g_{\rho \sigma} \dot{\gamma}_s^\sigma    (I_1^\rho)'(s) \right] \\
		& + \left( v^\lambda \partial_\lambda g_{\rho \sigma}(\gamma_s(t)) \dot{\gamma}_s^\sigma \dot{\gamma}_s^\nu + g_{\rho \sigma}(\gamma_s(t)) \dot{v}^\sigma \dot{\gamma}_s^\nu + g_{\rho \sigma}(\gamma_s(t)) \dot{\gamma}_s^\sigma \dot{v}^\nu \right) (I_1^\rho)'(s) \\
		& + g_{\rho \sigma}(\gamma_s(t)) \dot{\gamma}_s^\sigma \dot{\gamma}_s^\nu (I_1^\rho)''(s).
	\end{align*}
%
%
	Again, it follows from this computation, from the bound in \ref{eq:bound on I_1'(0)}, the facts that $\sup_{s \in (0,u)} |I_1(s)| \lesssim 1$, and $\sup_{s \in (0,u)} |I_1'(s)| \lesssim |\nabla V| + |V|$ that
	\begin{equation}\label{eq:bound on I_2'(0)}
		\sup_{s \in (0,u)} |(I_2^\nu)''(s)| \lesssim |V|^2 + |\nabla V|^2 + |\nabla^2 V|^2.
	\end{equation}
	Putting together the bounds obtained in \ref{eq:bound on I_1'(0)} and \ref{eq:bound on I_2'(0)} we discover that the remainder term in \ref{eq:Appendix Taylor expansion of the equation} satisfies
	$$|R_s(V)| \lesssim |V|^2 +  |\nabla V|^2 + |\nabla^2 V|^2 . $$
	Furthermore, integrating we find that
	\begin{equation}\label{eq:estimate on the remainder}
		\|R_s(V) \|_{L^2} \lesssim \|V\|_{W^{2,2}}^2.
	\end{equation}
	Next, we shall explain how to obtain the following estimate
	\begin{equation}\label{eq:estimate on the remainder (difference)}
		\| R_s(W_2) - R_s(W_1) \|_{L^2} \lesssim  \|W_1 + W_2\|_{W^{2,2}} \|W_2 - W_1 \|_{W^{2,2}}.
	\end{equation}
	To achieve this we write $R_s(V) = \frac{1}{s^2} \int_0^s (s-s')^2 \partial_{s'}^2 f_{s'}(V) ds'$ and compute
	\begin{align*}
		|R_s(W_2) - R_s(W_1)| & = \frac{1}{s^2} \Big\vert \int_0^s (s-s')^2 \left( \partial_{s'}^2 f_{s'}(W_2) - \partial_{s'}^2 f_{s'}(W_1) \right) ds' \Big\vert \\
		& = \frac{1}{s^2} \Big\vert \int_0^s (s-s')^2 \left( \partial_{s'}^2 \int_0^1 \partial_u f_{s'}(W_1 + u (W_2-W_1) ) du \right) ds' \Big\vert \\
		& \lesssim \Big\vert \int_0^1  \partial_u \int_0^s  \partial_{s'}^2  f_{s'}(W_1 + u (W_2-W_1) )  ds' du \Big\vert \\
		& = \Big\vert \int_0^1  \partial_u  \left[ \partial_{s'}  f_{s'}(W_1 + u (W_2-W_1) )  \right]_{s'=0}^{s'=s} du \Big\vert \\
		& \lesssim \sup_{s' \in [0,s]} \sup_{u \in [0,1]} \Big\vert  \partial_u \partial_{s'}  f_{s'}(W_1 + u (W_2-W_1) ) \Big\vert .
	\end{align*}
	We can now repeat the previous computation with $\gamma_s(t)=\exp_{\gamma_0(t)}(sW_1(t) + su (W_2(t)-W_1(t)) )$ using similar local coordinates. This immediately yields the bound in equation \ref{eq:estimate on the remainder (difference)}.

		\section{Maxwell's problem}
		
		It is an interesting question to ask for which configurations of points can the number of geodesics which we produce be maximized or minimized. Recall that in our case such geodesics come from critical points of $\phi$ which is the electric potential generated by a collection of point charges in either $\mathbb{R}^3$ or $\mathbb{T}_\Lambda$. In this manner, the critical points of $\phi$ corresponds to the location where the electric field $E=\nabla \phi$ vanishes.
		
		In fact, the question of what is the maximum number of critical points already goes back to Maxwell in his 1873 ``Treatise on Electricity and Magnetism'' \cite{Maxwell}. In it, Maxwell conjectures that the maxmimum number critical points is at most $(n-1)^2$. Surprisingly, very little research has been carried out in this direction, see however the very interesting article \cite{Gabrielov} where the authors compute upper bounds on the number of electrostatic points. Other relevant literature on the topic can be found in \cite{Killian,Tsai2} and references therein.

\end{document}